\RequirePackage{fix-cm}
\documentclass[smallextended]{svjour3}       % onecolumn (second format)
\smartqed  % flush right qed marks, e.g. at end of proof
\usepackage{graphicx}
\smartqed
\usepackage{booktabs}
\usepackage{graphicx}
\usepackage{subfigure}
\usepackage{underscore}
\usepackage{hyperref}	
\usepackage{lineno}
\usepackage{lipsum}
\usepackage{comment}
\hypersetup{colorlinks,%				% Reference color setup	
	linkcolor=blue,%
	citecolor=blue}

\journalname{}
\usepackage{mathptmx}      % use Times fonts if available on your TeX system
\usepackage{graphicx,epstopdf}
\usepackage{bm}
\usepackage{latexsym,enumerate}
\usepackage{mathrsfs}
\usepackage{amsfonts}
\usepackage{amssymb,amsmath}
\usepackage{epsfig}       %for Postscript files
\usepackage{subfigure}    %for parallel figures
\usepackage{graphicx}
\usepackage{float}
\usepackage{multirow}
\usepackage{array}
\usepackage{color}
\usepackage{listings}
\usepackage{caption}
\DeclareMathOperator*{\argmin}{argmin}
\usepackage[ruled]{algorithm2e}

\numberwithin{lemma}{section}
\numberwithin{theorem}{section}
\numberwithin{corollary}{section}
\numberwithin{definition}{section}
\numberwithin{property}{section}
\numberwithin{remark}{section}
\numberwithin{example}{section}
\numberwithin{figure}{section}
\numberwithin{remark}{section}
\newtheorem{assumption}{Assumption}
\begin{document}
%\linenumbers
\title{An Anderson-accelerated stochastic extragradient method for stochastic variational inequalities

\thanks{This paper is supported in part by National Natural Science Foundation of China grants (12425115, 12271127, 62176073) and Hong Kong Research Grant Council grants (PolyU15300124, PolyU15300625).}}

\titlerunning{An Anderson-accelerated stochastic extragradient method for stochastic VIs}        % if too long for running head

\author{ Xin Qu        \and
        Wei Bian   \and
        Xiaojun Chen
}


\institute{Xin Qu \at
              Department of Applied Mathematics, The Hong Kong Polytechnic University, Hung Hom, Kowloon, Hong Kong, China 
            \\  \email{hitxin.qu@connect.polyu.hk}           %  \\
           \and
          Wei Bian (Corresponding author)\at
              School of Mathematics, Harbin Institute of Technology, Harbin, China
              \\  \email{bianweilvse520@163.com}  
               \and
             Xiaojun Chen \at
              Department of Applied Mathematics, The Hong Kong Polytechnic University, Hung Hom, Kowloon, Hong Kong, China
              \\  \email{xiaojun.chen@polyu.edu.hk}  
}

\date{Received: date / Accepted: date}

\maketitle

\begin{abstract}
In this paper, we propose an Anderson-accelerated stochastic extragradient algorithm for solving a class of stochastic variational inequalities,
by incorporating Anderson acceleration into the stochastic extragradient method under a stochastic approximation framework.
A key challenge in our setting is that the pseudomonotonicity assumption is only imposed on the expectation of the stochastic operator, rather than on the 
individual stochastic operator itself and the sample averages utilized in the algorithm. We prove that, despite the lack of pseudomonotonicity in the sampled operators, the sequence generated by the proposed algorithm converges almost surely to a solution of the stochastic variational inequality problem. Additionally, we establish the sublinear convergence rate of the proposed algorithm in terms of the mean residual function, along with its optimal oracle complexity. Finally, we validate the effectiveness of the proposed algorithm through numerical experiments.
\end{abstract}

\section{Introduction}
Let $X\subseteq\mathbb{R}^n$ be a closed convex set and $(\Xi,\mathcal{G})$ be a measurable space. Suppose $T:\Xi\times\mathbb{R}^n\rightarrow\mathbb{R}^n$ is a measurable operator, and $\xi:\Omega\rightarrow\Xi$ is a random variable defined on a probability space $(\Omega, \mathcal{F}, \mathbb{P}).$ 
In this paper, we focus on the following stochastic variational inequality (SVI): find an $x^*\in X$ such that
\begin{equation}\label{ssvi}
	\langle \mathbb{E}[T(\xi,x^*)],x-x^*\rangle\geq0,\quad\forall x\in X,
\end{equation}
which has many applications in practice \cite{jadamba2015variational,lin2010stochastic,ravat2011characterization}. 
In particular, when the expectation can be evaluated explicitly, problem \eqref{ssvi} becomes a deterministic variational inequality (VI). For notational simplicity, we denote $H(x):=\mathbb{E}[T(\xi, x)]$, and use ${\rm SVI}(X, H)$ to represent problem \eqref{ssvi}. The solution set of ${\rm SVI}(X, H)$ is denoted by ${\rm SOL}(X, H)$, and we assume that ${\rm SOL}(X, H) \neq \emptyset$. 

\subsection{Stochastic approximation}
To address decision-making problems that involve future uncertainty, the stochastic version of variational inequalities (VIs) has been developed.
Stochastic variational inequalities (SVIs) are a powerful mathematical framework that generalizes deterministic VIs to incorporate randomness and uncertainty. 
%In many cases, VIs emerge as expectations of fundamental stochastic optimization problems.

%This formulation is referred to as the expected value approach to SVIs, originally introduced in the foundational work of King and {\color{red}Rockafellar (1993)\cite{}.} 

%If the operator $H(\cdot):=\mathbb{E}[T(\xi,\cdot)]$ defined in \eqref{ssvi} can be calculated accurately, then ${\rm SVI}(X, H)$ reduces to the deterministic VI problem, which can be solved using conventional techniques.
%However, in practical scenarios, $H$ is typically unavailable, either due to the computational complexity of evaluating the integral or because $H$ itself depends on solving another embedded subproblem. Consequently, solving SVIs often necessitates stochastic methods that rely on random samples of the underlying operator $T(x,\xi)$ rather than its exact expectation.
Since calculating expectations requires multi-dimensional integration, many existing numerical methods for the deterministic problem are typically not directly applicable to the stochastic problem ${\rm SVI}(X, H)$. 
To overcome this difficulty, stochastic approximation (SA) and the sample average approximation (SAA) are commonly used. 
The strategy of using the SAA method to handle SVIs is to discretize them into a sequence of deterministic VIs by generating independent identically distributed samples \cite{SAA1,SAA2,SAA3}.
Specifically, G$\ddot{\rm{u}}$rkan et al. introduced a sample-path approach in \cite{SAA1}, which incorporates SAA as a specific instance for solving the SVIs. Subsequently, Xu \cite{SAA3} utilized the SAA techniques to directly address the SVIs, obtaining some convergence results. Wang et al. \cite{SAA2} combined the regularized gap function with SAA techniques to develop a descent method for solving SVIs.
The resulting SVIs remain challenging to solve, particularly with a larger sample size.
The SA techniques can be used to reduce computational costs and storage.

The SA method, first proposed by Robbins and Monro \cite{robbins1951stochastic}, was originally developed for stochastic root-finding problems. This approach involves substituting the exact mean operator $H$ with a random sample of $T$ during the iteration process.
The first SA-based projection method for SVIs is derived from the combination of the classical projection scheme and SA techniques, as introduced in \cite{jiang2008stochastic}. 
This method ensures convergence under the assumptions of Lipschitz continuity and strong monotonicity of $H$.
Yousefian et al. \cite{yousefian2015self} extended the iterative scheme from \cite{jiang2008stochastic} to solve Cartesian SVIs. The convergence conditions in \cite{yousefian2015self} still require the operator $H$ to be strongly monotone and Lipschitz continuous.

Subsequently, some scholars weakened the strong monotonicity of the operator to monotonicity. Yousefian et al.  \cite{yousefian2017smoothing} proposed a regularized smoothing algorithm to solve monotone SVIs. 
Koshal et al. \cite{koshal2012regularized} introduced a stochastic iterative Tikhonov regularization method for SVIs, which requires the operator $H$ to be monotone and Lipschitz continuous for convergence.
Iusem et al. \cite{iusem2017extragradient} introduced a stochastic extragradient (SEG) method for solving ${\rm SVI}(X, H)$. The structure of the algorithm in \cite{iusem2017extragradient} is as follows:
\begin{equation*} 
	\quad\quad	\quad \quad\quad\quad	\quad \quad
	\left\{\begin{aligned}
		 & x_{k+0.5}:= P_{X}\left(x_k-t\frac{1}{S_k}\sum_{j=1}^{S_k}T(\xi_j^k, x_k)\right) \\ 
		 & x_{k+1}:= P_{X}\left(x_k-t\frac{1}{S_k}\sum_{j=1}^{S_k}T(\eta_j^k, x_{k+0.5})\right),
		 \end{aligned}\right. 
	\quad \quad\quad\quad	\quad \quad\text{(SEG)} 
\end{equation*}
where $t\in\left(0,\frac{1}{\sqrt{6}L}\right)$, $L$ is the Lipschitz constant of $H$, $\{S_k\}$ is the sample rate, $P_X:\mathbb{R}^n\rightarrow X$ is the Euclidean projection operator onto $X$, and $\{\xi_j^k\}_{j=1}^{S_k},  \{\eta_j^k\}_{j=1}^{S_k}$ are random samples. 
For this algorithm to achieve almost sure convergence, the operator $H$ needs to be pseudomonotone.

Subsequently, Iusem et al. \cite{iusem2019variance} proposed a dynamic stochastic extragradient method for SVIs, establishing the first robust SA algorithm that guarantees convergence without requiring prior knowledge of the Lipschitz constant.
Based on this pioneering work, Zhang et al. \cite{zhang2019infeasible} introduced an infeasible projection algorithm with line search for pseudomonotone SVIs, achieving the same iteration complexity but higher oracle complexity than \cite{iusem2019variance}. However, unlike \cite{iusem2019variance}, each iteration in \cite{zhang2019infeasible} requires only one projection and one sample family, significantly simplifying the implementation. To reduce the computational cost, Yang  et al. \cite{sref10} proposed a variance-based subgradient extragradient algorithm with line search for SVIs. The approach constructs a half-space at each iteration and replaces the second projection onto the feasible set with a projection onto this half-space. There are also several variants of variance-based extragradient algorithms with line search for solving SVIs explored in references \cite{kannan2019optimal,li2023improved,xiao2021unified}.
\subsection{Anderson acceleration for fixed point problems}\label{section1.2}
We know that the ${\rm SVI}(X, H)$ is equivalent to the following fixed point problem \cite{ref18}:
\begin{equation*}
	x=G(x):= P_{X}(x-tH(x))
\end{equation*}
with $t>0$. The fixed point iteration is an effective method for solving fixed point problems, which is given by $$x_{k+1}=G(x_k).$$
Anderson acceleration is a popular technique used to accelerate the convergence of the fixed point iteration. The algorithm framework is shown in Anderson(1) algorithm as follows.

\vspace*{3mm}
\begin{algorithm}[H]\normalsize
	%\captionsetup{format=Anderson(m)}
	\renewcommand{\thealgocf}{}
	\caption{Anderson(1)}\label{Alg}
	Choose $x_0\in \mathbb{R}^n$. Set $x_1=G(x_0)$ and $F_0=G(x_0)-x_0$.
	\\ \textbf{for} $k=1,2,\ldots$, \textbf{do}
	\\ \quad set $F_k=G(x_k)-x_k$;
	\\ \quad set
	\begin{align*}
		\begin{split}
			a_k=\frac{F_k^{{\rm T}}\left(F_k-F_{k-1}\right)}{\|F_k-F_{k-1}\|^2}
		\end{split}
	\end{align*}
	\\	\quad and
	\begin{align*}
		x_{k+1}=(1-a_k)G(x_{k})+a_kG(x_{k-1}).
	\end{align*}
	\\ \textbf {end for}
\end{algorithm}
\vspace*{3mm}
Anderson acceleration, originally proposed in \cite{ref38} for solving partial differential equations, has become one of the most widely used acceleration techniques. Anderson acceleration works by combining previous iterates, with the mixing coefficients determined through the solution of a subproblem. Anderson acceleration has been widely applied in various fields, including electronic structure computations \cite{ref38,ref39}, machine learning \cite{ref40}, fluid dynamics \cite{fluid} and geometrical optimization \cite{Peng}.

Although the practical effectiveness of Anderson acceleration has been acknowledged for many years, a rigorous convergence theory has just been established recently.
Toth and Kelley \cite{ref42} provided the first proof of local r-linear convergence for Anderson(m) algorithm when $G$ is Lipschitz continuously differentiable and contractive. In particular, when $m=1$, Anderson(1) algorithm exhibits q-linear convergence.
In 2019, Chen and Kelley \cite{ref39} relaxed the requirements on $G$, demonstrating that the same results hold as long as $G$ is a continuously differentiable operator. Evans et al. \cite{evans} provided the theoretical justification that Anderson acceleration improves the convergence rate of contractive fixed-point iterations by introducing higher order terms.
Pollock and Rebholz \cite{Pollock} presented a one-step analysis of Anderson acceleration, revealing how the resulting residual bounds balance the contributions of higher and lower order terms in both contractive and noncontractive settings.
When the fixed point operator is nonsmooth, some scholars have also conducted some research.
Bian et al. \cite{ref43} established the q-linear convergence of the Anderson(1) algorithm for general nonsmooth fixed point problems in Hilbert spaces, and also proved r-linear convergence of Anderson(m) for a particular nonsmooth operator. Later, Bian and Chen \cite{addition} showed that Anderson(1) achieves q-linear convergence for the composite max fixed point problem, with a smaller q-factor compared to earlier results. In addition, there are some variants of Anderson acceleration that guarantee the global convergence for fixed point problems, but are with no convergence rates \cite{Fu,ref44}.

For deterministic pseudomonotone VIs,  Qu et al. \cite{Qu} developed an extragradient based method with Anderson(1) acceleration to accelerate the convergence. Motivated by this idea and the notable performance of Anderson acceleration methods across various algorithms, we design a new stochastic accelerated algorithm for solving ${\rm SVI}(X, H)$, which combines Anderson(1) acceleration with the SEG method. In contrast to the pseudomonotone operator required in \cite{Qu}, we only assume the expected operator in ${\rm SVI}(X, H)$ to be pseudomonotone. As a result, a key challenge in the analysis of our algorithm arises from the fact that 
the individual stochastic operator and the sample average operators employed in the proposed algorithm do not preserve this pseudomonotonicity property. 
Despite this challenge, we are still able to prove the  convergence of the generated sequence to a solution of ${\rm SVI}(X, H)$. Thus, the contributions of this paper include the following three aspects.
\begin{enumerate}[{\rm (i)}]
\item 
We develop an accelerated stochastic algorithm for solving ${\rm SVI}\left(X, H\right)$, which only requires pseudomonotonicity of the expected operator, while allowing the individual stochastic operator and sample averages to be non-pseudomonotone.
 This algorithm combines the SEG method with Anderson acceleration.
	 Notably, the implementation of a stochastic line search eliminates the need for prior knowledge of the Lipschitz constant, making the approach highly practical for complex stochastic environments.
	\item  We demonstrate that the sequence generated by the proposed stochastic algorithm converges almost surely to a solution of ${\rm SVI}\left(X, H\right)$. Furthermore, we show that the proposed stochastic algorithm achieves a sublinear convergence rate of $\mathcal{O}\left(N^{-1}\right)$ in terms of the mean residual function, as well as an optimal oracle complexity.
	\item The numerical experiments ultimately demonstrate the superior performance of the proposed stochastic algorithm compared to the SEG method.
	Furthermore, for some applications of portfolio selection with real data, the solution found by the proposed stochastic algorithm performs better than the naive evenly weighted portfolio methods.
\end{enumerate}

This paper is organized as follows. 
In Section \ref{s2}, we briefly review some related concepts and preliminary results used in this paper. In Section \ref{section3}, we propose an accelerated algorithm for solving ${\rm SVI}(X, H)$ by using the Anderson(1) and the 
SEG methods. Additionally, we analyze the sequence convergence of the algorithm and provide the convergence rate as well as oracle complexity. Finally, we validate the performance of the proposed algorithm through three numerical experiments in Section \ref{section4}.
\section{Preliminaries}\label{s2}
We denote by $\mathbb{R}_+^{n}$ the nonnegative orthant in $\mathbb{R}^{n}$.
Let $\|x\|$ denote the Euclidean norm of vector $x\in\mathbb{R}^n$, and let $\|A\|$ and  $\|A\|_F$ represent the $\ell_{2}$-norm and the Frobenius norm of a matrix $A\in\mathbb{R}^{m\times n}$, respectively.
Given a closed set $X\subseteq\mathbb{R}^n$ and a point $x\in\mathbb{R}^n$, we use the notation $P_{X}(x):=\argmin_{y\in X}\|y-x\|^2$.
Let $\mathbb{N}$ be the set of all positive integers and $\mathbb{N}_0:=\mathbb{N}\cup\{0\}$. 
Let $\textbf{e}\in\mathbb{R}^n$
denote the vector of all ones, i.e., $\textbf{e}=(1,1,\ldots,1)^{{\rm T}}\in \mathbb{R}^n$.
For any $a\in\mathbb{R},$ $\lceil a \rceil$ represents the smallest integer not smaller than $a.$ 
Given two sequences $\{b_k\}$ and $\{d_k\}$, $b_k=\mathcal{O}(d_k)$ means that there exists a constant $c>0$ such that $|b_k|\leq c|d_k|$ for all $k$.
Given a random variable $\xi$ and a $\sigma$-algebra $\mathcal{F}$,
we denote the expectation of $\xi$ by $\mathbb{E}[\xi]$ and the conditional expectation of $\xi$ by $\mathbb{E}[\xi|\mathcal{F}]$.
We use $[m]:=\{1,2,\cdots,m\}$ for $m\in \mathbb{N}$, and 
denote the $\sigma$-algebra generated by the random variables $\xi^1,\cdots,\xi^k$ as $\sigma(\xi^1,\cdots,\xi^k)$. Given the random variable $\xi$ and $p\geq1,$ $\|\xi\|_p$ is the $L^p$-norm of $\xi$ and $\|\xi|\mathcal{F}\|_p:=(\mathbb{E}[\|\xi\|^p|\mathcal{F}])^{\frac{1}{p}}$ is the $L^p$-norm of $\xi$ conditional to the $\sigma$-algebra $\mathcal{F}.$ We write $\xi\in\mathcal{F}$ if $\xi$ is $\mathcal{F}$-measurable, i.i.d. for ‘ independent identically distributed’, and a.s. for  ‘almost surely’.

For any $x\in\mathbb{R}^n$ and $t>0,$ we denote the residual function by
\begin{equation}\label{residual}
	r_t(x):=\|P_{X}(x-tH(x))-x\|.
\end{equation}

The following lemma demonstrates that the residual function can be employed to quantify the solutions of ${\rm SVI}(X, H)$.
\begin{lemma}\label{equ}{\rm\cite[Proposition 1.5.8]{ref18}}
	$x^*\in$ ${\rm SOL}(X,H)$ if and only if it is a solution of $r_t(x)=0$, where $t$ can be any positive number.
\end{lemma}

The following definitions and lemmas will be applied in the subsequent analysis.
\begin{definition}\cite{def12}\label{def}
	Let $X$ be a subset of $\mathbb{R}^n$ and $H: \mathbb{R}^n \to \mathbb{R}^n$ an operator.
	Then
	\begin{enumerate}[{\rm (i)}]
		\item \label{defm1} $H$ is said to be monotone on $X$, if for any $x,y\in X$ it holds $$\langle H(x)-H(y), x-y\rangle\geq0;$$
		
		\item \label{defm2} $H$ is said to be pseudomonotone on $X$, if for any $x,y\in X$ it holds $$\langle H(x), y-x\rangle\geq0\Rightarrow \langle H(y),y-x\rangle\geq0.$$
	\end{enumerate}
\end{definition}
\begin{remark}
	From Definition \ref{def}, the following relation holds
	$$\text{monotone}\Rightarrow \text{pseudomonotone}.$$
\end{remark}
\begin{lemma}\cite{ref}\label{lee}
	For any $x\in\mathbb{R}^n$, the following statements hold.
	\begin{enumerate}[{\rm (i)}]
		\item \label{lee0}$\|P_{X}(x)-P_{X}(y)\|\leq\|x-y\|, ~~\forall y\in\mathbb{R}^n;$
		
		\
		
		\item \label{lee1} $\|P_{X}(x)-P_{X}(y)\|^2\leq\langle P_{X}(x)-P_{X}(y), x-y\rangle, ~~\forall y\in\mathbb{R}^n;$
		
		\
		
		\item \label{lee2} $\langle x-P_{X}(x), y-P_{X}(x)\rangle\leq 0, ~~\forall y\in X.$
	\end{enumerate}
\end{lemma}
\begin{lemma}\cite{ref45}\label{lemm1}
	For any $x\in\mathbb{R}^n$ and $t_1\geq t_2>0$, the following inequalities hold:
	$$\frac{\|x-P_{X}(x-t_1H(x))\|}{t_1}\leq\frac{\|x-P_{X}(x-t_2H(x))\|}{t_2},$$
	$$\|x-P_{X}(x-t_2H(x))\|\leq\|x-P_{X}(x-t_1H(x))\|.$$
\end{lemma}
\begin{definition}\cite{durrett2019probability}
	If $\{x_k\}$ is a sequence with
	\begin{enumerate}[{\rm (i)}]
		\item $\mathbb{E}[\|x_k\|]<\infty$ for all $k$,
		
		\

		\item %$x_k$ is adapted to $\mathcal{F}_k,$ 
		$x_k$ is $\mathcal{F}_k$-measurable for each $k$,
		
		\
		
		\item $\mathbb{E}[x_{k+1}|\mathcal{F}_k]=x_k$ for all $k$,
	\end{enumerate}
	then $\{x_k\}$ is said to be a martingale. 
\end{definition}
\section{Proposed stochastic algorithm and its convergence analysis}\label{section3}
In this section, we first propose a stochastic accelerated algorithm for solving ${\rm SVI}(X,H)$, which is based on the Anderson acceleration method and the SEG method, combined with the SA approach. Next, we provide some convergence analysis on the proposed algorithm.
\subsection{Framework of the proposed stochastic algorithm}\label{sec4.11}
Before presenting the framework of the proposed stochastic algorithm, we first introduce some key notations.

We define the oracle error $\epsilon: \Xi \times \mathbb{R}^n\rightarrow\mathbb{R}^n$ by
\begin{equation}\label{green1}
	\epsilon(\xi,x):=T(\xi,x)-H(x), \quad \forall \xi\in\Xi, \quad\forall x\in\mathbb{R}^n.
\end{equation}
Motivated by (SEG), we select the sample rate $\{S_k\}$. At each iteration of the proposed algorithm, we generate two independent samples $\{\xi_j^k\}_{j=1}^{S_k}$ and $\{\eta_j^k\}_{j=1}^{S_k}$, and then approximate $H(x)$ at a given $x$ by the following two empirical estimators: $$H_{\xi^k}(x)=\frac{1}{S_k}\sum_{j=1}^{S_k} T(\xi_j^k, x)\quad {\rm and} \quad H_{\eta^k}(x)=\frac{1}{S_k}\sum_{j=1}^{S_k} T(\eta_j^k, x).$$

Algorithm \ref{alg3} presents the proposed algorithm, called the Anderson(1)-SEG algorithm, in which the line search framework from \cite{sref10} is utilized in Step 2.
\begin{algorithm}\normalsize
	%\caption{Extragradient algorithm with Anderson(1) acceleration (EG-Anderson(1) algorithm)}
	\renewcommand{\thealgocf}{1}
	\caption{Anderson(1)-SEG}\label{alg3}
	\SetAlgoNoLine
	\textbf{Initialization:} 
	% \\ \textbf {Step 1}: Compute $x_1=\tilde{G}(x_0)$ and $$.
	Choose $x_0\in X$, $\nu\geq0, \gamma\in(0,1]$, $\rho\in(0,1)$, $\mu\in(0,\frac{\sqrt{3}}{3})$, $\tau>$ $\frac{1}{2}$,  $\theta_0=1$ and the sample rate $\{S_k\}$. Give a sufficiently large $M>0$.
	\\ \textbf{for} $k=0,1,2,\ldots$, \textbf{do}
	%\\ \quad \quad Let $t_k=\gamma$.
	%	\\ \quad \quad	{\color{red}Generate $\{\xi^k_j\}_{j=1}^{S_k}$ and $\{\eta^k_j\}_{j=1}^{S_k}$ of the random variable $\xi$.}
	\\ \quad \quad \textbf{Step 1:} Generate $\{\xi^k_j\}_{j=1}^{S_k}$ and $\{\eta^k_j\}_{j=1}^{S_k}$ of the random variable $\xi$. Let \\  \quad \quad \quad \quad\quad $t_k=\gamma$. 	%\\  \quad \quad \quad \quad\quad 
	Compute 
	%$\hat{H}(\xi^k, x_k)=\frac{1}{S_k}\sum_{j=1}^{S_k}T(\xi_j^k,x_k)$ and
	%\\ \quad \quad \quad \quad \quad 
	$F_{\gamma, \xi^{k}}=P_{X}\left(x_k- \gamma H_{\xi^k}(x_k)\right)-x_k.$
	\\ \quad \quad \quad \quad \quad \textbf{If} $F_{\gamma, \xi^{k}}=\mathbf{0}$, set 
	\begin{equation}\label{rege}
		x_{k+1}=x_k, \quad \theta_{k+1}=\theta_k. 
	\end{equation}
	\\ \quad \quad \quad \quad \quad \textbf{Otherwise}, go to \textbf{Step 2}.
	\\ \quad \quad \textbf{Step 2:} Compute $y_{k+0.5}=P_{X}\left(x_k-t_kH_{\xi^k}( x_k)\right)$.
	\\ \quad \quad \quad \quad \quad \textbf{If}
	\begin{equation}\label{lins}
		t_k\left\|H_{\xi^k}( x_k)-H_{\xi^k}(y_{k+0.5})\right\|
		\leq \mu\|x_k-y_{k+0.5}\|, 
	\end{equation}
	\quad \quad \quad \quad \quad \quad go to \textbf{Step 3}.
	\\ \quad \quad \quad \quad \quad \textbf{Otherwise}, set $t_k=\rho t_k$ and repeat \textbf{Step 2}.
	\\ \quad \quad \textbf{Step 3:} Compute	$ F_{t_k, \xi^k}=y_{k+0.5}-x_k $. Compute
	%$\hat{H}(\eta^{S_k}, y_{k+0.5}) =$ 
	%	\\ \quad \quad \quad \quad \quad \quad$ \frac{1}{S_k}\sum_{j=1}^{S_k}T(\eta_j^k,y_{k+0.5})$,
	$$y_{k+1}=P_{X}\left(x_k-t_kH_{\eta^k}( y_{k+0.5})\right)
	\quad {\rm and } \quad
	\tilde{F}_{t_k,\eta^k}=y_{k+1}-x_k.$$
	\\ \quad \quad \quad\quad\quad  \textbf{If} $\left\| \tilde{F}_{t_k,\eta^k}\right\|<\min\left\{\left\|F_{t_k,\xi^k}\right\|, \nu\theta_k^{-\tau}\right\}$,
	set
	\begin{equation}\label{alphas}
		\alpha_{k}= \frac{\left\langle \tilde{F}_{t_k,\eta^k}, \tilde{F}_{t_k,\eta^k}-F_{t_k,\xi^k}\right\rangle}{\left\|\tilde{F}_{t_k,\eta^k}-F_{t_k,\xi^k}\right\|^2}.
	\end{equation}
	%\\ \quad \quad \quad \quad \quad and go to Step 4-4a.
	\\ \quad \quad \quad \quad \quad \textbf{Otherwise}, set $\alpha_{k}=M+1$.
	\\ \quad \quad \textbf{Step 4:} \textbf {If} $|\alpha_{k}|\leq M$, set
	\begin{align}\label{SAnderson}
		x_{k+1}=	\alpha_{k}x_k+(1-\alpha_{k})y_{k+1},\quad \theta_{k+1}=\theta_k+1.
	\end{align}
	\\ \quad \quad \quad \quad \quad \textbf{Otherwise}, set
	\begin{equation}\label{SEG1}
		x_{k+1}=y_{k+1},\quad \theta_{k+1}=\theta_k.
	\end{equation}
	\\ \textbf {end for} 
\end{algorithm}
We will analyze the convergence properties of the Anderson(1)-SEG algorithm under the following assumptions.		\begin{assumption}\label{Minty}
			The mean operator $H:\mathbb{R}^n\rightarrow\mathbb{R}^n$ is pseudomonotone on $X$.
		\end{assumption}
		\begin{remark}\cite{kannan2014pseudomonotone} If $T(\xi,\cdot)$ is pseudomonotone on $X$ in an almost sure sense, then $H(\cdot):=\mathbb{E}[T(\xi,\cdot)]$ is pseudomonotone on $X$. However, the reverse implication does not hold in general. 
           For example, consider the mapping $$T(\xi, x) = \xi |x|,$$ where $x \in \mathbb{R}$ and $\xi$ is a random variable following a standard normal distribution. It is easy to see that the expected mapping $H(x)=\mathbb{E}[T(\xi, x)] = |x|\cdot\mathbb{E}[\xi] = 0$ is pseudomonotone. However, $T(\xi, x)$ is not pseudomonotone on $X=[-1,1]$ for $\xi\neq 0$.
		\end{remark}
		\begin{assumption}\label{a22}
			\begin{enumerate}[{\rm (i)}]
				\item There exists a measurable function $L:\Xi\rightarrow\mathbb{R}^n$ such that
				$L(\xi)\ge1$ holds with probability $1$, and for any $x,y\in\mathbb{R}^n$,
				\[
				\|T(\xi,x)-T(\xi,y)\|\le L(\xi)\|x-y\|,
				\]
				and moreover,  for $p\ge2$,
				\[
				\mathbb{E}\big[L(\xi)^{2p}\big]<\infty.
				\]
				
				\item There exist a vector $\check{x}\in\mathbb{R}^n$ and a number $q\ge2$ such that $$\mathbb{E}\big[\|T(\xi,\check{x})\|^{2q}\big]<\infty.$$
				
			\end{enumerate}
		\end{assumption}

\begin{assumption}\label{a3}
   In Algorithm \ref{alg3}, $\{\xi^k_j\}_{j=1}^{S_k}$ and $\{\eta^k_j\}_{j=1}^{S_k}$ are two mutually independent sets of i.i.d. samples drawn from $\xi$. Moreover, $\sum_{k=0}^{\infty}\frac{1}{\sqrt{S_k}}<\infty$.
\end{assumption}
\begin{remark}\label{lip}
Denote $\bar{L}:=\mathbb{E}[L(\xi)]$ and $\bar{L}_p:=(\mathbb{E}[L(\xi)^p])^{\frac{1}{p}}+\bar{L}$ for any $p\geq1$.
If Assumption \ref{a22} holds, then 
\begin{enumerate}[{\rm (i)}]
\item \label{lip1}by Jensen's inequality, for any $x,y\in\mathbb{R}^n$, we obtain
\begin{equation*}
\begin{split}
\|H(x)-H(y)\|=&\left\|\mathbb{E}[T(\xi,x)]-\mathbb{E}[T(\xi,y)]\right\|=\|\mathbb{E}[T(\xi,x)-T(\xi,y)]\|
\\ \leq & \mathbb{E}[\|T(\xi,x)-T(\xi,y)\|]\leq  \mathbb{E}[L(\xi)\|x-y\|]= \bar{L}\|x-y\|,
\end{split}
\end{equation*}
which means that $H$ is $\bar{L}$-Lipschitz continuous;
\item  \label{lip2}by Minkowski's inequality and (i), for any $x, x_*\in\mathbb{R}^n$, we have
\begin{equation*}
	\begin{split}
&	\left(\mathbb{E}[\|\epsilon(\xi,x)\|^p]\right)^{\frac{1}{p}} 
\\ = &	\left(\mathbb{E}[\|T(\xi, x)-H(x)\|^p]\right)^{\frac{1}{p}}
\\ \leq &
\left(\mathbb{E}[\|T(\xi,x)-T(\xi,x_*)\|^p]\right)^{\frac{1}{p}}+
\left(\mathbb{E}[\|T(\xi, x_*)-H(x_*)\|^p]\right)^{\frac{1}{p}}
\\ &+\|H(x)-H(x_*)\|
\\ \leq & (\mathbb{E}[L(\xi)^p])^{\frac{1}{p}}\|x-x_*\|+  \left(\mathbb{E}[\|\epsilon(\xi,x_*)\|^p]\right)^{\frac{1}{p}}+\bar{L}\|x-x_*\|
\\ = & \bar{L}_p\|x-x_*\|+\left(\mathbb{E}[\|\epsilon(\xi,x_*)\|^p]\right)^{\frac{1}{p}} , \quad \forall p\geq1.			
\end{split}
\end{equation*}
\end{enumerate}
\end{remark}
For the convenience of subsequent analysis, we denote $\xi^k:=\{\xi^k_j\}_{j=1}^{S_k}$ and $\eta^k:=\{\eta^k_j\}_{j=1}^{S_k}$. Define $$\mathcal{F}_0=\sigma(x_0), ~~~ \hat{\mathcal{F}}_0=\sigma(x_0, \xi^0),$$ 
and 
$$\mathcal{F}_k=\sigma(x_0, \xi^0,\cdots,\xi^{k-1},\eta^0,\cdots,\eta^{k-1}),~~~ \mathcal{\hat{F}}_k=\sigma(x_0, \xi^0,\cdots,\xi^{k},\eta^0,\cdots,\eta^{k-1})$$ for $k\geq1$ as the $\sigma$-algebras related to the generation of $y_{k+0.5}$, $y_{k}$ and $x_k.$
It can be observed that $\mathcal{F}_k\subset\mathcal{\hat{F}}_k,$ $x_k, y_k\in\mathcal{F}_k, y_{k+0.5}\in\mathcal{\hat{F}}_k$ and $y_{k+0.5}\notin\mathcal{F}_k$. 
Recalling \eqref{green1} and the Anderson(1)-SEG algorithm, we denote the oracle errors by
\begin{equation*}
	\begin{cases}
		\hat{\epsilon}_1^k := \frac{1}{S_k}\sum_{j=1}^{S_k}\epsilon(\xi^k_j, x_k)=
		H_{\xi^k}(x_k) - H(x_k), \\ 
		\hat{\epsilon}_2^k := \frac{1}{S_k}\sum_{j=1}^{S_k}\epsilon(\eta^k_j, y_{k+0.5}) = 
		H_{\eta^k}( y_{k+0.5})- H(y_{k+0.5}), \\
		\hat{\epsilon}_3^k := \frac{1}{S_k}\sum_{j=1}^{S_k}\epsilon(\xi^k_j, y_{k+0.5}) =
		H_{\xi^k}(y_{k+0.5}) - H(y_{k+0.5}).
	\end{cases}
\end{equation*}
Then $y_{k+0.5}$ and $y_{k+1}$ in the Anderson(1)-SEG algorithm can be written as follows:
\begin{equation}\label{redefine}
	\begin{cases}
		y_{k+0.5} = P_{X}\left(x_k - t_k(H(x_k) + \hat{\epsilon}_1^k)\right), \\
		y_{k+1} = P_{X}\left(x_k - t_k(H(y_{k+0.5}) + \hat{\epsilon}_2^k)\right).
	\end{cases}
\end{equation}
\begin{remark}
	Since $\xi^{k}$ is independent of $\mathcal{F}_k$ and $x_{k}\in\mathcal{F}_k$, we have
	\begin{equation*}
		\begin{split}
			\mathbb{E}[\hat{\epsilon}_1^k|\mathcal{F}_k]= & \mathbb{E}\left[\frac{1}{S_k}\sum_{j=1}^{S_k}T(\xi_j^k, x_k)-H(x_k)|\mathcal{F}_k\right]
			\\ = & \frac{1}{S_k}\sum_{j=1}^{S_k}\mathbb{E}[T(\xi_j^k, x_k)|\mathcal{F}_k]-H(x_k)
			\\=&H(x_k)-H(x_k)=\mathbf{0}.
		\end{split}
	\end{equation*}
	Similarly,  from the fact that $\eta^{k}$ is independent of $\mathcal{\hat{F}}_k$ and $y_{k+0.5}\in\hat{\mathcal{F}}_k$, we obtain $	\mathbb{E}\left[\hat{\epsilon}_2^k|\mathcal{\hat{F}}_k\right]= \mathbf{0}$. In view of $\mathcal{F}_k\subset\mathcal{\hat{F}}_k$, we conclude that  $\mathbb{E}\left[\hat{\epsilon}_2^k|\mathcal{F}_k\right]=\mathbb{E}\left[\mathbb{E}[\hat{\epsilon}_2^k|\hat{\mathcal{F}}_k]|\mathcal{F}_k\right]=\mathbf{0}$.
	Thus, $\{\hat{\epsilon}_1^k\}$ and $\{\hat{\epsilon}_2^k\}$ are the martingale difference sequences with respect to $\mathcal{F}_k$.
\end{remark}
The following result on super-martingale convergence will be utilized, serving as a key instrument in demonstrating the convergence of SA methods.
\begin{theorem}\cite{robbins1971convergence}\label{a.s.convergence}
	Let $\mathcal{F}_0\subset\mathcal{F}_1\subset\dots$ a sequence of sub-$\sigma$-algebras of $\mathcal{F}$. For each $k=0,1,\dots,$ let $\{v_k\}, \{u_k\},\{a_k\} ~and ~\{b_k\}$ be four sequences of random variables. Suppose 
	\begin{enumerate}[{\rm (i)}]
		\item the random variables $\{v_k\}, \{u_k\},\{a_k\}$ and $\{b_k\}$ are nonnegative $\mathcal{F}_k$-measurable;
		\item a.s. for each $k$, $$\mathbb{E}[v_{k+1}|\mathcal{F}_k]\leq(1+a_k)v_k-u_k+b_k;$$
		\item  a.s. $\sum_{k=0}^{\infty} a_k<\infty,$ $\sum_{k=0}^{\infty} b_k<\infty$.
	\end{enumerate}
	Then, a.s. $\{v_k\}$ converges and $\sum_{k=0}^{\infty} u_k<\infty$.
\end{theorem}
\subsection{Convergence analysis}\label{sec4.2}
In this subsection, we will present some convergence properties of the Anderson(1)-SEG algorithm.
We divide the iterations of the Anderson(1)-SEG algorithm into three subsets
\begin{equation*}
	K_{SRE} = \{u_0, u_1, \cdots\}, \quad K_{SAA} = \{k_0, k_1, \cdots\}\quad\mbox{and}\quad K_{SEG} = \{l_0, l_1, \cdots\},
\end{equation*}
where $K_{SRE}$ consists of the iterations defined by \eqref{rege}, $K_{SAA}$ comprises iterations defined by \eqref{SAnderson}, and $K_{SEG}$ contains the remaining iterations defined by \eqref{SEG1}. 

We begin by proving that the Anderson(1)-SEG algorithm is well-defined.
\begin{lemma}
	Suppose that Assumption \ref{a22} holds.
	For each iteration $k$ where line search is performed,
	the line search \eqref{lins} in the Anderson(1)-SEG algorithm terminates in a finite number of runs.
\end{lemma}
{\it \textbf{Proof}} \quad
In the Anderson(1)-SEG algorithm, the updated form of $t_k$ can be re-expressed as $t_k=\gamma \rho^{m_k},$ where $m_k$ is defined as the smallest nonnegative integer $m$ that satisfies the following inequality:
\begin{equation*}
	\gamma \rho^{m}\left\| H_{\xi^k}( x_k)-H_{\xi^k}(y^{(m)}_{k+0.5})\right\|\leq \mu\left\|x_k-y^{(m)}_{k+0.5}\right\|,
\end{equation*}
where $y^{(m)}_{k+0.5}=P_{X}\left(x_k-\gamma \rho^m H_{\xi^k}(x_k)\right)$. 
We prove the lemma by contradiction, assuming that the line search \eqref{lins} does not terminate within a finite number of iterations. This means that, for every $m$, we have
\begin{equation}\label{milk}
	\gamma \rho^{m}\left\| H_{\xi^k}( x_k)-H_{\xi^k}( y^{(m)}_{k+0.5})\right\|> \mu\left\|x_k-y^{(m)}_{k+0.5}\right\|.	
\end{equation}
We will discuss it in two cases.
\begin{enumerate}[{\rm (i)}]
	\item If $x_k\in X$, then from the definition of $y^{(m)}_{k+0.5}$ and the fact that $\rho\in(0,1)$, it follows that $\lim_{m\rightarrow\infty}\left\|y^{(m)}_{k+0.5}-x_k\right\|=0$. Combining Assumption \ref{a22}, we deduce that 
	\begin{equation}\label{flower}
		\lim_{m\rightarrow\infty}\left\| H_{\xi^k}( x_k)-H_{\xi^k}( y^{(m)}_{k+0.5})\right\|=0.
	\end{equation}
	From \eqref{milk} and Lemma \ref{lemm1}, we have
	\begin{equation*}
		\left\| H_{\xi^k}( x_k)-H_{\xi^k}( y^{(m)}_{k+0.5})\right\|>\frac{ \mu}{\gamma \rho^{m}}\left\|x_k-y^{(m)}_{k+0.5}\right\|\geq \frac{\mu}{\gamma}\left\|F_{\gamma, \xi^k}\right\|>0.
	\end{equation*}
	Taking the limit $m\rightarrow\infty$ in the above inequality, we can find a contradiction with \eqref{flower}.
	\item If $x_k\notin X,$ then
	\begin{equation}\label{musi1}
		\lim_{m\rightarrow\infty}\left\|x_k-y^{(m)}_{k+0.5}\right\|>0,
	\end{equation}
	and
	\begin{equation}\label{musi2}
		\lim_{m\rightarrow\infty} \gamma \rho^{m}\left\| H_{\xi^k}( x_k)-H_{\xi^k}( y^{(m)}_{k+0.5})\right\|=0.
	\end{equation}
	Letting $m\rightarrow\infty$ in \eqref{milk}, we obtain a contradiction with \eqref{musi1} and \eqref{musi2}.
\end{enumerate}
Thus, the line search \eqref{lins} is well-defined.  \qed

The Anderson(1)-SEG algorithm employs a dynamic sampled SA line search method to address the situation in which the Lipschitz constant is unknown.

The following lemma provides a lower bound for the step size.

\begin{lemma}\label{tab222}
	Suppose that Assumptions \ref{a22} and \ref{a3} hold. Let $\{t_k\}$ be the stepsize sequence generated by the Anderson(1)-SEG algorithm and $\tilde{L}_k:=\frac{1}{S_k}\sum_{j=1}^{S_k}L(\xi_j^k)$. Then we have that a.s.
	%$t_k\geq \min \left\{\frac{\rho\mu}{\hat{L}_k(\xi^k)},\gamma\right\}$ 
	$t_k\tilde{L}_k\geq\min\left\{\rho\mu,\gamma\right\}$ and $\sqrt{\left(\mathbb{E}[t_k^2|\mathcal{F}_k]\right)}\cdot\sqrt{\left(\mathbb{E}[L(\xi)^2]\right)}\geq\min\left\{\rho\mu,\gamma\right\}$.
\end{lemma}
{\it \textbf{Proof}} \quad 
We first discuss the lower bound of $t_k$ by dividing it into two cases.
\begin{enumerate}[{\rm (i)}]
	\item
	If line search  \eqref{lins} is not performed, or $\gamma$ satisfies \eqref{lins} during the line search, then $t_k=\gamma.$
	\item Otherwise, $t_k<\gamma$. Let $\hat{y}_{k+0.5}:= P_{X}\left(x_{k}-t_{k}\rho^{-1}H_{\xi^k}(x_k)\right)$.
	By $t_k<\gamma$, $t_k<t_k\rho^{-1}$ and Lemma \ref{lemm1}, we have
	$$\left\|\hat{y}_{k+0.5}-x_k\right\|\geq\left\|y_{k+0.5}-x_k\right\|\geq \frac{t_k}{\gamma}\left\|F_{\gamma,\xi^k}\right\|>0.$$
	Then we find
	\begin{equation}\label{coffee1}
		t_k\rho^{-1}\left\|H_{\xi^k}( x_k)-H_{\xi^k}( \hat{y}_{k+0.5})\right\| > \mu\left\|x_k-\hat{y}_{k+0.5}\right\|>0.
	\end{equation}
	By Assumption \ref{a22}-(i), we have
	\begin{equation}\label{coffee2}
		\begin{split}
			\left\|H_{\xi^k}( x_k)-H_{\xi^k}( \hat{y}_{k+0.5})\right\|=&\left\| \frac{1}{S_k}\sum_{j=1}^{S_k}T(\xi_j^k,x_k)-\frac{1}{S_k} \sum_{j=1}^{S_k}T(\xi_j^k,\hat{y}_{k+0.5})\right\|
			\\ \leq &\tilde{L}_k\|x_k-\hat{y}_{k+0.5}\|.
		\end{split}
	\end{equation}
	Combining \eqref{coffee1} and \eqref{coffee2}, we obtain $t_k\geq \frac{\rho\mu}{\tilde{L}_k}.$ 
\end{enumerate}
Thus, we conclude that $t_k\geq \min \left\{\frac{\rho\mu}{\tilde{L}_k},\gamma\right\}$.

Since a.s. $L(\xi)\geq1$, we have that a.s. $t_k\tilde{L}_k\geq\min\left\{\rho\mu,\gamma\right\}$. Combining it with the H\"{o}lder's inequality, we deduce that
\begin{equation*}
	\begin{split}
		\min\left\{\rho\mu,\gamma\right\}&\leq\mathbb{E}\left[t_k\tilde{L}_k|\mathcal{F}_k\right]
		\leq \sqrt{\left(\mathbb{E}[t_k^2|\mathcal{F}_k]\right)} \sqrt{\left(\mathbb{E}[\tilde{L}^2_k|\mathcal{F}_k]\right)}
		\\ & \leq \sqrt{\left(\mathbb{E}[t_k^2|\mathcal{F}_k]\right)}\sqrt{\left(\frac{1}{S_k}\sum_{j=1}^{S_k}\mathbb{E}[L(\xi_j^k)^2|\mathcal{F}_k]\right)}=\sqrt{\left(\mathbb{E}[t_k^2|\mathcal{F}_k]\right)}\cdot\sqrt{\left(\mathbb{E}[L(\xi)^2]\right)},
	\end{split}
\end{equation*} 
where the last equality follows from Assumption \ref{a3}.
\qed

The recursive relations given by the following two lemmas play a crucial role in the main convergence results.
\begin{lemma}\label{lemm1s}
	Suppose that Assumptions \ref{Minty} and \ref{a22} hold.
	Let $\{x_k\}$, $\{y_{k+0.5}\}$ and $\{y_{k+1}\}$ be the sequences generated by the Anderson(1)-SEG algorithm. Then for all $x^*\in{\rm SOL}\left(X, H\right)$, we have
	\begin{equation}\label{SQ1}
		\begin{split}
			\left\|y_{k+1}-x^*\right\|^2 \leq & \|x_k-x^*\|^2-\frac{(1-3\mu^2)t_k^2}{2\gamma^2}r_{\gamma}(x_k)^2+ (1-3\mu^2)\gamma^2\left\|\hat{\epsilon}_1^k\right\|^2
			\\ &+ 3\gamma^2\left\|\hat{\epsilon}_2^k\right\|^2+3\gamma^2\left\|\hat{\epsilon}_3^k\right\|^2+2\left\langle t_k\hat{\epsilon}_2^k, x^*-y_{k+0.5}\right\rangle.
		\end{split}
	\end{equation}
\end{lemma}
{\it \textbf{Proof}} \quad 
By Assumption \ref{Minty} and $x^*\in{\rm SOL}(X, H)$, we have $$\left\langle H(y_{k+0.5}), y_{k+0.5}-x^*\right\rangle\geq 0,$$ 
which implies
\begin{equation}\label{cheng}
	\left	\langle t_k H(y_{k+0.5}), x^*-y_{k+1}\right\rangle \leq \left\langle t_k H(y_{k+0.5}), y_{k+0.5}-y_{k+1}\right\rangle.
\end{equation}
By the definition of $y_{k+0.5}$ and Lemma \ref{lee}-\eqref{lee2}, we find
$$\left\langle x_k-t_kH_{\xi^k}( x_k)-y_{k+0.5}, y_{k+1}-y_{k+0.5}\right\rangle\leq0.$$
This implies
\begin{equation}\label{nans}
	\begin{split}
		&	\left\langle x_k-t_kH(y_{k+0.5})-y_{k+0.5}, y_{k+1}-y_{k+0.5}\right\rangle
		\\ \leq &\left \langle t_kH_{\xi^k}( x_k)-t_kH(y_{k+0.5}), y_{k+1}-y_{k+0.5}\right\rangle.
	\end{split}
\end{equation}
Let $z_k:=x_k-t_kH_{\eta^{k}}( y_{k+0.5})$. Using Lemma \ref{lee}-\eqref{lee2}, \eqref{cheng} and \eqref{nans}, we have
\allowdisplaybreaks
\begin{equation}\label{rep}
	\begin{split}
		\left\|y_{k+1}-x^*\right\|^2 
		=~ &\left \|P_{X}(z_k)-x^*\right\|^2
		\\ = ~& \left\|z_k-x^*\right\|^2-\left\|z_k-P_{X}(z_k)\right\|^2+2\left\langle P_{X}(z_k)-z_k, P_{X}(z_k)-x^*\right\rangle
		\\ \leq ~& \left\|z_k-x^*\right\|^2-\left\|z_k-P_{X}(z_k)\right\|^2
		\\ =~ & \left\|x_k-t_kH_{\eta^k}( y_{k+0.5})-x^*\right\|^2-\left\|x_k-t_kH_{\eta^k}(y_{k+0.5})-y_{k+1}\right\|^2
		\\  =~ & \left\|x_k-x^*\right\|^2-\left\|x_k-y_{k+1}\right\|^2+2\left\langle t_kH_{\eta^k}(y_{k+0.5}), x^*-y_{k+1}\right\rangle
		\\ = ~ & \left\|x_k-x^*\right\|^2-\left\|x_k-y_{k+1}\right\|^2+2\left\langle t_kH( y_{k+0.5}), x^*-y_{k+1}\right\rangle 
		\\ & + 2\left\langle t_k\hat{\epsilon}_2^k, x^*-y_{k+1}\right\rangle
		\\ \overset{\eqref{cheng}}\leq & \left\|x_k-x^*\right\|^2-\left\|x_k-y_{k+1}\right\|^2+2 \left\langle t_k H(y_{k+0.5}), y_{k+0.5}-y_{k+1}\right\rangle
		\\ & + 2\left\langle t_k\hat{\epsilon}_2^k, x^*-y_{k+1}\right\rangle
		\\ = ~ & \left\|x_k-x^*\right\|^2-\left\|x_k-y_{k+0.5}\right\|^2-\left\|y_{k+0.5}-y_{k+1}\right\|^2
		\\ & + 2\left\langle x_k-t_k H(y_{k+0.5})-y_{k+0.5}, y_{k+1}-y_{k+0.5}\right\rangle+ 2\left\langle t_k\hat{\epsilon}_2^k, x^*-y_{k+1}\right\rangle
		\\ \overset{\eqref{nans}} \leq  & \|x_k-x^*\|^2-\|x_k-y_{k+0.5}\|^2-\|y_{k+0.5}-y_{k+1}\|^2
		\\ & + 2 \left\langle t_kH_{\xi^{k}}( x_k)-t_kH(y_{k+0.5}), y_{k+1}-y_{k+0.5}\right\rangle + 2\left\langle t_k\hat{\epsilon}_2^k, x^*-y_{k+1}\right\rangle
		\\ =~ &  \|x_k-x^*\|^2-\|x_k-y_{k+0.5}\|^2-\|y_{k+0.5}-y_{k+1}\|^2
		\\ & + 2\left\langle t_kH_{\xi^k}(x_k)-t_kH_{\eta^k}(y_{k+0.5}), y_{k+1}-y_{k+0.5}\right\rangle+ 2\left\langle t_k\hat{\epsilon}_2^k, x^*-y_{k+0.5}\right\rangle.
	\end{split}
\end{equation}
Next, we estimate the term
$ 2\left\langle t_kH_{\xi^k}(x_k)-t_kH_{\eta^k}( y_{k+0.5}), y_{k+1}-y_{k+0.5}\right\rangle$ in \eqref{rep}.
By the Cauchy-Schwarz inequality and the Young's inequality, we have
\begin{equation}\label{schsw}
	\begin{split}
		& 2\left\langle t_kH_{\xi^k}(x_k)-t_kH_{\eta^k}( y_{k+0.5}), y_{k+1}-y_{k+0.5}\right\rangle
		%\\ \leq ~~& 2t_k\|\hat{H}(\xi^{S_k},x_k)-\hat{H}(\eta^{S_k}, y_{k+0.5})\|\|y_{k+1}-y_{k+0.5}\|
		\\ \leq ~~& t_k^2\left\|H_{\xi^k}(x_k)-H_{\eta^k}( y_{k+0.5})\right\|^2+ \left\|y_{k+1}-y_{k+0.5}\right\|^2
		\\ = ~~&t_k^2\|H_{\xi^k}(x_k)-H_{\xi^k}( y_{k+0.5})+H_{\xi^k}( y_{k+0.5})-H(y_{k+0.5})
		\\ &+H(y_{k+0.5})-H_{\eta^k}( y_{k+0.5})\|^2+ \left\|y_{k+1}-y_{k+0.5}\right\|^2
		\\ \leq ~~& 3t_k^2\left\|H_{\xi^k}(x_k)-H_{\xi^k}( y_{k+0.5})\right\|^2+3\gamma^2\left\|\hat{\epsilon}_2^k\right\|^2+3\gamma^2\left\|\hat{\epsilon}_3^k\right\|^2
		\\ & + \left\|y_{k+1}-y_{k+0.5}\right\|^2
		\\ \overset{\eqref{lins}}\leq ~~& 3\mu^2\|x_k-y_{k+0.5}\|^2+3\gamma^2\left\|\hat{\epsilon}_2^k\right\|^2+3\gamma^2\left\|\hat{\epsilon}_3^k\right\|^2+ \left\|y_{k+1}-y_{k+0.5}\right\|^2.
	\end{split}
\end{equation}
Substituting \eqref{schsw} into \eqref{rep}, we deduce that
\begin{equation}\label{lopp}
	\begin{split}
		\left\|y_{k+1}-x^*\right\|^2 \leq & \left\|x_k-x^*\right\|^2-(1-3\mu^2)\left\|x_k-y_{k+0.5}\right\|^2+ 3\gamma^2\left\|\hat{\epsilon}_2^k\right\|^2+3\gamma^2\left\|\hat{\epsilon}_3^k\right\|^2
		\\ &+2\left\langle t_k\hat{\epsilon}_2^k, x^*-y_{k+0.5}\right\rangle.
	\end{split}
\end{equation}
Recalling the function $r_t(x):=\left\|P_{X}(x-tH(x))-x\right\|$ defined in \eqref{residual}, and combining this with $t_k\leq\gamma$ and Lemma \ref{lemm1}, we obtain
\begin{equation*}
	\begin{split}
		t_k^2r_{\gamma}(x_k)^2\leq & \gamma^2 r_{t_k}(x_k)^2=\gamma^2 \left\|P_{X}(x_k-t_kH(x_k))-x_k\right\|^2
		\\ \leq & 2\gamma^2\left\|x_k-y_{k+0.5}\right\|^2+2\gamma^2\left\|P_{X}(x_k-t_kH(x_k))-P_{X}\left(x_k-t_kH_{\xi^k}(x_k)\right)\right\|^2
		\\ \leq & 2\gamma^2\left\|x_k-y_{k+0.5}\right\|^2+2\gamma^4\left\|\hat{\epsilon}_1^k\right\|^2,
	\end{split}
\end{equation*}
which implies 
\begin{equation}\label{lpo}
	\left\|x_k-y_{k+0.5}\right\|^2\geq\frac{t_k^2}{2\gamma^2}r_{\gamma}(x_k)^2-\gamma^2\left\|\hat{\epsilon}_1^k\right\|^2.
\end{equation}
Plugging \eqref{lpo} into \eqref{lopp} yields
\begin{equation*}
	\begin{split}
		\left\|y_{k+1}-x^*\right\|^2 \leq & \left\|x_k-x^*\right\|^2-\frac{(1-3\mu^2)t_k^2}{2\gamma^2}r_{\gamma}(x_k)^2+ (1-3\mu^2)\gamma^2\left\|\hat{\epsilon}_1^k\right\|^2
		\\ &+ 3\gamma^2\left\|\hat{\epsilon}_2^k\right\|^2+3\gamma^2\left\|\hat{\epsilon}_3^k\right\|^2+2\left\langle t_k\hat{\epsilon}_2^k, x^*-y_{k+0.5}\right\rangle.
	\end{split}
\end{equation*}
\qed

\begin{lemma}\label{lem4.66}
	Suppose that Assumptions \ref{Minty} and \ref{a22} hold.
	Let $\{x_k\}$ and $\{y_{k+0.5}\}$ be the sequences generated by the Anderson(1)-SEG algorithm. Then, for any $\hat{\beta}\in(0,\frac{1}{2}]$  and $x^*\in{\rm SOL}(X, H)$, we have
	\begin{equation}\label{lem4.6}
		\begin{split}
			\left\|x_{k+1}-x^*\right\|^2\leq & \left\|x_k-x^*\right\|^2-\frac{(1-3\mu^2)\hat{\beta}t_{k}^2}{2\gamma^2}r_{\gamma}(x_{k})^2+(1-3\mu^2)(M+1)\gamma^2\left\|\hat{\epsilon}_1^{k}\right\|^2
			\\ & +3\gamma^2(M+1)\left\|\hat{\epsilon}_2^{k}\right\|^2+3\gamma^2(M+1)\left\|\hat{\epsilon}_3^{k}\right\|^2
			\\ & +2(M+1)\left|\left\langle x^*-y_{k+0.5}, t_{k}\hat{\epsilon}_2^{k}\right\rangle\right|+\varrho_{k},
		\end{split}
	\end{equation}
	where 
\begin{eqnarray}\label{varsigma}
		\varrho_k=	\left\{\begin{aligned}
			& M(M+1)\nu^2(1+i)^{-2\tau},  ~~~&{\rm if} ~~~k\in K_{SAA}:=\{k_i, i=1,2,\ldots\},
			\\& 0, ~~~~~~~~~~~~~~~~~~~~~~~~~~~~~&{\rm if} ~~~k\in K_{SRE}\cup K_{SEG}.
		\end{aligned}\right.
	\end{eqnarray}
\end{lemma}
	{\it \textbf{Proof}} \quad We will discuss this by dividing it into three cases.
\begin{enumerate}[{\rm (i)}]
	\item 	For $u_m\in K_{SRE},$ we know that $\left\|F_{\gamma,\xi^{u_m}}\right\|=\left\|P_{X}(x_{u_m}-\gamma H_{\xi^{u_m}}( x_{u_m}))-x_{u_m}\right\|=0$ and $x_{u_m+1}=x_{u_m}.$ Then we have
	\begin{equation*}
		\begin{split}
			r_{\gamma}(x_{u_m})^2= &  \left\|P_{X}\left(x_{u_m}-\gamma H(x_{u_m})\right)-x_{u_m}\right\|^2
			\\ \leq & 2\left\|F_{\gamma,\xi^{u_m}}\right\|^2+2\left\|P_{X}(x_{u_m}-\gamma H(x_{u_m}))
			-P_{X}\left(x_{u_m}-\gamma H_{\xi^{u_m}}(x_{u_m})\right)\right\|^2
			\\ \leq & 2\gamma^2\left\|\hat{\epsilon}_1^{u_m}\right\|^2.
		\end{split}
	\end{equation*}
	Combining the above inequality with $t_{u_m}=\gamma$ and $\mu\in(0,\frac{\sqrt{3}}{3})$, we obtain $$-\frac{(1-3\mu^2)t_{u_m}^2}{2\gamma^2}r_{\gamma}(x_{u_m})^2+ (1-3\mu^2)\gamma^2\left\|\hat{\epsilon}_1^{u_m}\right\|^2\geq 0.$$
	Therefore, we conclude that
	\begin{equation*}
		\begin{split}
			\left\|x_{u_m+1}-x^*\right\|^2 
			= & \left\|x_{u_m}-x^*\right\|^2
			\\	 \leq & \left\|x_{u_m}-x^*\right\|^2-\frac{(1-3\mu^2)t_{u_m}^2}{2\gamma^2}r_{\gamma}(x_{u_m})^2+ (1-3\mu^2)\gamma^2\left\|\hat{\epsilon}_1^{u_m}\right\|^2
			\\ &+ 3\gamma^2\left\|\hat{\epsilon}_2^{u_m}\right\|^2+3\gamma^2\left\|\hat{\epsilon}_3^{u_m}\right\|^2+2\left|\left\langle t_{u_m}\hat{\epsilon}_2^{u_m}, x^*-y_{u_m+0.5}\right\rangle\right|,
		\end{split}
	\end{equation*}
	which implies that \eqref{lem4.6} holds in this case with any $\hat{\beta}\leq1$.
	\item 	For $k_i\in K_{SAA}$, let $\beta_{k_i}:=1-\alpha_{k_i} $.
	By the definition of $x_{k_i+1}$ 
	in \eqref{SAnderson}, we obtain
	\begin{equation}\label{qww}
		\begin{split}
			\left\|x_{k_i+1}-x^*\right\|^2 
			= ~& \left \|\alpha_{k_i}  x_{k_i}+\beta_{k_i}y_{k_i+1}-x^*\right\|^2
			\\ = ~& \alpha_{k_i}^2\left\|x_{k_i}-x^*\right\|^2+\beta_{k_i}^2\left\|y_{k_i+1}-x^*\right\|^2
			+2\alpha_{k_i}\beta_{k_i}\left\langle x_{k_i}-x^*, y_{k_i+1}-x^*\right\rangle
			\\  = ~& \alpha_{k_i}^2\left\|x_{k_i}-x^*\right\|^2+\beta_{k_i}^2\left\|y_{k_i+1}-x^*\right\|^2
			\\ &+\alpha_{k_i}\beta_{k_i} \left(\left\|x_{k_i}-x^*\right\|^2+\left\|y_{k_i+1}-x^*\right\|^2-\left\|x_{k_i}-y_{k_i+1}\right\|^2\right)
			\\ =~ & \alpha_{k_i}\left\|x_{k_i}-x^*\right\|^2+\beta_{k_i}\left\|y_{k_i+1}-x^*\right\|^2
			-\alpha_{k_i}\beta_{k_i}\left\|x_{k_i}-y_{k_i+1}\right\|^2.
		\end{split}
	\end{equation}
	Since  $\left\|\tilde{F}_{t_{k_i}, \eta^{k_i}}\right\|<\left\|F_{t_{k_i}, \xi^{k_i}}\right\|$ when $k_i\in K_{SAA}$, then for any $\hat{\beta}\in(0,\frac{1}{2}]$, we have
	\begin{equation*}
		\begin{split}
			\hat{\beta}\left\|\tilde{F}_{t_{k_i}, \eta^{k_i}}\right\|^2+(1-2\hat{\beta})\left\langle F_{t_{k_i}, \xi^{k_i}}, \tilde{F}_{t_{k_i}, \eta^{k_i}}\right\rangle
			<  (1-\hat{\beta})\left\|F_{t_{k_i}, \xi^{k_i}}\right\|^2,
		\end{split}
	\end{equation*}
	which implies
	\begin{equation*}
		\begin{split}
			\left \langle \tilde{F}_{t_{k_i}, \eta^{k_i}}-F_{t_{k_i}, \xi^{k_i}}, \tilde{F}_{t_{k_i}, \eta^{k_i}}\right\rangle
			<  (1-\hat{\beta})\left\|\tilde{F}_{t_{k_i}, \eta^{k_i}}-F_{t_{k_i}, \xi^{k_i}}\right\|^2.
		\end{split}
	\end{equation*}
	By \eqref{alphas}, we have
	$\alpha_{k_i}<1-\hat{\beta}.$ Thus $\beta_{k_i}=1-\alpha_{k_i}>\hat{\beta}.$
	Substituting \eqref{SQ1} into \eqref{qww} and using $\mu\in(0,\frac{\sqrt{3}}{3})$, $|\alpha_{k_i}|\leq M,$ $\hat{\beta}<\beta_{k_i}\leq M+1$, $\left\|x_{k_i}-y_{k_i+1}\right\|=\left\|\tilde{F}_{t_{k_i},\eta^{k_i}}\right\|< \nu \theta_{k_i}^{-\tau}=\nu (\theta_0+i)^{-\tau}=\nu (1+i)^{-\tau}$ and \eqref{varsigma}, we deduce that
	\begin{equation*}
		\begin{split}
			&	\left\|x_{k_i+1}-x^*\right\|^2\
			\\ \leq & \left\|x_{k_i}-x^*\right\|^2-\frac{(1-3\mu^2)\hat{\beta}t_{k_i}^2}{2\gamma^2}r_{\gamma}(x_{k_i})^2+(1-3\mu^2)(M+1)\gamma^2\left\|\hat{\epsilon}_1^{k_i}\right\|^2
			\\ & +3\gamma^2(M+1)\left\|\hat{\epsilon}_2^{k_i}\right\|^2+3\gamma^2(M+1)\left\|\hat{\epsilon}_3^{k_i}\right\|^2
			\\ & +2(M+1)\left|\left\langle x^*-y_{k_i+0.5}, t_{k_i}\hat{\epsilon}_2^{k_i}\right\rangle\right|+\varrho_{k_i},
		\end{split}
	\end{equation*}
	which implies that \eqref{lem4.6} holds.
	\item For $l_j\in K_{SEG}$, by Lemma \ref{lemm1s}, we deduce that
	\begin{equation*}
		\begin{split}
			\left\|x_{l_j+1}-x^*\right\|^2 = & \left\|y_{l_j+1}-x^*\right\|^2
			\\ \leq & \left\|x_{l_j}-x^*\right\|^2-\frac{(1-3\mu^2)t_{l_j}^2}{2\gamma^2}r_{\gamma}(x_{l_j})^2+ (1-3\mu^2)\gamma^2\left\|\hat{\epsilon}_1^{l_j}\right\|^2
			\\ &+ 3\gamma^2\left\|\hat{\epsilon}_2^{l_j}\right\|^2+3\gamma^2\left\|\hat{\epsilon}_3^{l_j}\right\|^2+2\left|\left\langle t_{l_j}\hat{\epsilon}_2^{l_j}, x^*-y_{l_j+0.5}\right\rangle\right|,
		\end{split}
	\end{equation*}
\end{enumerate}
which implies that \eqref{lem4.6} holds with any $\hat{\beta}\leq1$ for this case.
\qed

Lemma \ref{lem4.66} indicates that both bounding $\mathbb{E}\left[\left|\langle x^*-y_{k+0.5}, t_{k}\hat{\epsilon}_2^{k}\rangle\right||\mathcal{F}_k\right]$ and the second order moments of the sequences $\{\hat{\epsilon}_1^k\}$, $\{\hat{\epsilon}_2^k\}$ and $\{\hat{\epsilon}_3^k\}$ are essential for establishing the convergence properties of the Anderson(1)-SEG algorithm, and we will do these estimations in Lemma \ref{QQ4}.

The following lemma is mainly used to control the oracle errors given above.
\begin{lemma}[Burkholder-Davis-Gundy inequality]\cite{davis2011integral}\label{le21}
	For any $p\geq2$ and $N\in\mathbb{N}$, there exists a $B_p>0$ such that for any vector-valued martingale $(u_i)_{i=1}^N$ taking values in $\mathbb{R}^n$ with $u_0=0,$ it holds that
	\begin{equation*}
		\begin{split}
			(\mathbb{E}[\|u_{N}\|^p])^{\frac{1}{p}}\leq & \left(\mathbb{E}\left[\left(\sup_{i\leq N}\|u_i\|\right)^p\right]\right)^{\frac{1}{p}}\leq B_p\left(\mathbb{E}\left[\left(\sum_{i=1}^N\|u_i-u_{i-1}\|^2\right)^{\frac{p}{2}}\right]\right)^{\frac{1}{p}}
			\\	\leq & B_p\sqrt{\sum_{i=1}^N\left(\mathbb{E}[\|u_i-u_{i-1}\|^p\right])^{\frac{2}{p}}}.
		\end{split}
	\end{equation*}
\end{lemma}
Recalling the definition of $\epsilon(\xi,x)$ in \eqref{green1}. 
The estimation of
$\left(\mathbb{E}\left[\left\|\frac{1}{N}\sum_{j=1}^N\epsilon(\xi_j,x)\right\|^p\right]\right)^{\frac{1}{p}}$ in \cite{iusem2017extragradient}
% $\left|\|\hat{\epsilon}(\xi^N, x)\|\right|_p$
is provided under some certain conditions.  We state this result in the form of the following lemma.
\begin{lemma} \label{QQ3}
	\textcolor[rgb]{0.00,0.00,1.00}{{\rm\cite[Lemma 3.12]{iusem2017extragradient}}}
	Suppose that Assumption \ref{a22} holds.
	Given $N\in\mathbb{N}$, let $\{\xi_j\}_{j=1}^N$ be an i.i.d. sample of $\xi$. 
	Then, for any $p\geq2$, $x, x_*\in\mathbb{R}^n$, we have
	$$\left(\mathbb{E}\left[\left\|\frac{1}{N}\sum_{j=1}^N\epsilon(\xi_j,x)\right\|^p\right]\right)^{\frac{1}{p}}\leq B_p \frac{\delta_p(x_*)\left(1+\|x-x_*\|\right)}{\sqrt{N}}$$and
	$$\left(\mathbb{E}\left[\left|\left\langle v, \frac{1}{N}\sum_{j=1}^N\epsilon(\xi_j,x) \right\rangle\right|^p\right]\right)^{\frac{1}{p}}\leq \|v\|B_p\frac{\delta_p(x_*)\left(1+\|x-x_*\|\right)}{\sqrt{N}},\quad\forall v\in\mathbb{R}^n,
	$$ 
	where $B_p$ is the same as in Lemma \ref{le21} and $\delta_p(x_*)=\max \left\{\left(\mathbb{E}[\|\epsilon(\xi,x_*)\|^p]\right)^{\frac{1}{p}}, \bar{L}_p\right\}$ with $\bar{L}_p$ defined in Remark \ref{lip}.
\end{lemma}	

The following lemma establishes an upper bound on the second order moment of the martingale difference sequences $\{\hat{\epsilon}_1^k\}$ and $\{\hat{\epsilon}_2^k\}$. While the prior works \cite{iusem2017extragradient,iusem2019variance} have proved similar bounds using the Burkholder-Davis-Gundy inequality, we will present a more direct characterization.

For any $x^*\in{\rm SOL}(X, H)$, we define 
\begin{equation}\label{dwa2} D_{k,p}(x^*):=\gamma\frac{1}{\sqrt{S_k}}B_p\delta_p(x^*)\quad \text{and}\quad I_{k,p}(x^*):=1+\gamma \bar{L}+D_{k,p}(x^*),
\end{equation}
where $B_p$, $\delta_p(x_*)$, and $\bar{L}$ are as defined in Lemma \ref{le21}, Lemma \ref{QQ3}, and Remark \ref{lip}, respectively.
%	\vspace{2em}
\begin{lemma}\label{QQ4}
	Suppose that Assumptions \ref{a22} and \ref{a3} hold. Then for any $k$, $x^*\in{\rm SOL}(X, H)$, $p=2q$ and $q\geq1$, it holds that
	\begin{enumerate}[{\rm (i)}]
		\item
		$\left(\mathbb{E}\left[\left\|\hat{\epsilon}_1^k\right\|^{2q}|\mathcal{F}_k\right]\right)^{\frac{1}{q}}\leq 2B_p^2 \dfrac{\delta_p(x^*)^2\left(1+\|x_k-x^*\|^2\right)}{S_k};$
		
		\
		
		\item $	\left(\mathbb{E}\left[\left\|\hat{\epsilon}_2^k\right\|^{2q}|\mathcal{F}_k\right]\right)^{\frac{1}{q}}\leq  \dfrac{2}{S_k}B_p^2\delta_p(x^*)^2\left(1+2I_{k,p}(x^*)^2\|x_k-x^*\|^2+2D_{k,p}(x^*)^2\right);$
		
		\
		
		\item $ \left(\mathbb{E}\left[\left|\left\langle x^*-y_{k+0.5}, t_k\hat{\epsilon}_2^k\right\rangle\right|^{q}|\mathcal{F}_k\right]\right)^{\frac{1}{q}}
		\leq  \left(\frac{1}{2} D_{k,p}(x^*)+2 D_{k,p}(x^*) I_{k,p}(x^*)^2\right)\|x_k-x^*\|^2
		\\ +\frac{1}{2} D_{k,p}(x^*) 
		I_{k,p}(x^*)^2
		+ D_{k,p}(x^*)^2+2 D_{k,p}(x^*)^3.$
	\end{enumerate}	
\end{lemma}

{\it \textbf{Proof}}\quad  We will prove the above three statements separately.
\begin{enumerate}[{\rm (i)}]
	\item Using Lemma \ref{QQ3}, $x_k\in\mathcal{F}_k$, and the independence of $\xi^{k}$ with $\mathcal{F}_k$, we get 
	\begin{equation}\label{QQ2}
		\left(\mathbb{E}\left[\left\|\hat{\epsilon}_1^k\right\|^{p}|\mathcal{F}_k\right]\right)^{\frac{1}{p}}
		=	\left(\mathbb{E}\left[\left\|\hat{\epsilon}_1^k\right\|^{p}\right]\right)^{\frac{1}{p}}\leq B_p \frac{\delta_p(x^*)\left(1+\|x_k-x^*\|\right)}{\sqrt{S_k}},
	\end{equation}
	which implies
	\begin{equation*}
		\begin{split}
			\left(\mathbb{E}\left[\left\|\hat{\epsilon}_1^k\right\|^{2q}|\mathcal{F}_k\right]\right)^{\frac{1}{q}}
			&=	\left(\mathbb{E}\left[\left\|\hat{\epsilon}_1^k\right\|^{p}|\mathcal{F}_k\right]\right)^{\frac{2}{p}}
			\\ &=\left(\mathbb{E}\left[\left\|\hat{\epsilon}_1^k\right\|^{p}\right]\right)^{\frac{2}{p}} \leq 2B_p^2 \frac{\delta_p(x^*)^2\left(1+\|x_k-x^*\|^2\right)}{S_k}.
		\end{split}
	\end{equation*}
	\item Recall that $y_{k+0.5}:=P_{X}\left(x_k-t_k\left(H(x_k)+\hat{\epsilon}_1^k\right)\right)$ in \eqref{redefine}. By Lemma \ref{equ} and $x^*\in{\rm SOL}(X, H)$, we have $x^*=P_{X}\left(x^*-t_kH(x^*)\right)$. 
	%From Jensen's Inequality and Assumption \ref{a2}, for any $x\in X$, we find $$\|H(x)\|=\|\mathbb{E}[T(\xi,x)]\|\leq\mathbb{E}[\|T(\xi,x)\|]\leq \hat{M}.$$ 
	Then, combining Lemma \ref{lee}-\eqref{lee0} with Remark \ref{lip}-\eqref{lip1} and $t_k\leq\gamma$, we conclude that
	\begin{equation}\label{QQ1}
		\begin{split}
			\|x^*-y_{k+0.5}\|=&\left\|P_{X}(x^*-t_kH(x^*))-P_{X}\left(x_k-t_k(H(x_k)+\hat{\epsilon}_1^k)\right)\right\|
			\\ \leq &  \|x_k-x^*\|+ t_k\|H(x_k)-H(x^*)\|+t_k\left\|\hat{\epsilon}_1^k\right\|
			\\ \leq &  \|x_k-x^*\| + \gamma \bar{L}\|x_k-x^*\| +\gamma\left\|\hat{\epsilon}_1^k\right\|,
		\end{split}
	\end{equation}
	where $\bar{L}=\mathbb{E}[L(\xi)]$. Taking $\left(\mathbb{E}[\|\cdot\|^p|\mathcal{F}_k]\right)^{\frac{1}{p}}$ in \eqref{QQ1} and using Minkowski's inequality, $x_k\in\mathcal{F}_k$, and \eqref{QQ2}, we deduce that
	\begin{equation}\label{yui}
		\begin{split}
			&       \left(\mathbb{E}\left[\|x^*-y_{k+0.5}\|^p|\mathcal{F}_k\right]\right)^\frac{1}{p}
			\\ \leq & \left(\mathbb{E}\left[\left(\|x_k-x^*\| + \gamma \bar{L}\|x_k-x^*\| +\gamma\left\|\hat{\epsilon}_1^k\right\|\right)^p|\mathcal{F}_k\right]\right)^\frac{1}{p}
			\\ \leq & (1+\gamma \bar{L})\|x_k-x^*\| +\gamma
			\left(\mathbb{E}\left[\left\|\hat{\epsilon}_1^k\right\|^p|\mathcal{F}_k\right]\right)^{\frac{1}{p}}
			\\ \leq & (1+\gamma \bar{L})\|x_k-x^*\| +\gamma B_p\frac{\delta_p(x^*)\left(1+\|x_k-x^*\|\right)}{\sqrt{S_k}}
			\\ = & \left(1+\gamma \bar{L}+\gamma\frac{1}{\sqrt{S_k}}B_p\delta_p(x^*)\right)\|x_k-x^*\|+\gamma\frac{1}{\sqrt{S_k}}B_p\delta_p(x^*)
			\\ = & I_{k,p}(x^*) \|x_k-x^*\|+D_{k,p}(x^*).
		\end{split}
	\end{equation}
	Then, we have
	\begin{equation}\label{QQ41}
		\begin{split}
			\left(\mathbb{E}\left[\|x^*-y_{k+0.5}\|^p|\mathcal{F}_k\right]\right)^\frac{2}{p}\leq  2I_{k,p}(x^*)^2\|x_k-x^*\|^2+2D_{k,p}(x^*)^2.
		\end{split}
	\end{equation}
	Since $\eta^{k}$ is independent of $\hat{\mathcal{F}}_k$ and $y_{k+0.5}\in\hat{\mathcal{F}}_k$, Lemma \ref{QQ3} implies that
	\begin{equation}\label{xw1}
		\left(\mathbb{E}\left[\left\|\hat{\epsilon}_2^k\right\|^p|\hat{\mathcal{F}}_k\right]\right)^{\frac{1}{p}}
		=	\left(\mathbb{E}\left[\left\|\hat{\epsilon}_2^k\right\|^p\right]\right)^{\frac{1}{p}}\leq B_p \dfrac{\delta_p(x^*)\left(1+\|y_{k+0.5}-x^*\|\right)}{\sqrt{S_k}}.
	\end{equation}
	%	$|\cdot|\mathcal{\hat{F}}_k|_q\leq|\cdot|\mathcal{\hat{F}}_k|_p$, $||\cdot|\mathcal{\hat{F}}_k|_q|\mathcal{F}_k|_q$$=|\cdot|\mathcal{F}_k|_q$, 
	Combining with Minkowski's inequality, \eqref{QQ41}, \eqref{xw1}, $\left(\mathbb{E}[\|\cdot\|^q|\mathcal{\hat{F}}_k]\right)^{\frac{1}{q}}\leq\left(\mathbb{E}[\|\cdot\|^p|\mathcal{\hat{F}}_k]\right)^{\frac{1}{p}}$ and
	$\left(\mathbb{E}\left[\left(\mathbb{E}[\|\cdot\|^q|\mathcal{\hat{F}}_k]\right)|\mathcal{F}_k\right]\right)^{\frac{1}{q}}=\left(\mathbb{E}[\|\cdot\|^q|\mathcal{F}_k]\right)^{\frac{1}{q}},$ we deduce that
	\begin{align*}
		\left(\mathbb{E}\left[\left\|\hat{\epsilon}_2^k\right\|^{2q}|\mathcal{F}_k\right]\right)^{\frac{1}{q}}=~&
		\left(\mathbb{E}\left[\left\|\hat{\epsilon}_2^k\right\|^{p}|\mathcal{F}_k\right]\right)^{\frac{2}{p}}	=\left(\mathbb{E}\left[\left(\mathbb{E}\left[\left\|\hat{\epsilon}_2^k\right\|^p|\mathcal{\hat{F}}_k\right]\right)|\mathcal{F}_k\right]\right)^{\frac{2}{p}}
		\\ \overset{\eqref{xw1}}\leq & \left(\mathbb{E}\left[\left(\frac{1}{\sqrt{S_k}}B_p\delta_p(x^*)\left(1+\|y_{k+0.5}-x^*\|\right)\right)^p|\mathcal{F}_k\right]^{\frac{1}{p}}\right)^2
		\\ \leq ~& \left(\frac{1}{\sqrt{S_k}}B_p\delta_p(x^*)+\frac{1}{\sqrt{S_k}}B_p\delta_p(x^*)\mathbb{E}\left[\|y_{k+0.5}-x^*\|^p|\mathcal{F}_k\right]^{\frac{1}{p}}\right)^2
		\\ \leq ~& \frac{2}{S_k}B_p^2\delta_p(x^*)^2\left(1+\left(\mathbb{E}\left[\|x^*-y_{k+0.5}\|^p|\mathcal{F}_k\right]\right)^\frac{2}{p}\right)
		\\ \overset{\eqref{QQ41}}\leq & \frac{2}{S_k}B_p^2\delta_p(x^*)^2\left(1+2I_{k,p}(x^*)^2\|x_k-x^*\|^2+2D_{k,p}(x^*)^2\right).
	\end{align*}
	\item By Lemma \ref{QQ3}, $y_{k+0.5}\in\hat{\mathcal{F}}_k$, \eqref{xw1} and $t_k\leq\gamma$, we deduce that
	\begin{equation}\label{QQ5}
		\begin{split}
			&	\left(\mathbb{E}\left[\left|\langle x^*-y_{k+0.5}, t_k\hat{\epsilon}_2^k\rangle\right|^p|\hat{\mathcal{F}}_k\right]\right)^{\frac{1}{p}}
			\\ \leq ~& \gamma\|x^*-y_{k+0.5}\|
			\left(\mathbb{E}\left[\left\|\hat{\epsilon}_2^k\right\|^p|\hat{\mathcal{F}}_k\right]\right)^{\frac{1}{p}}
			\\ \overset{\eqref{xw1}}\leq & \gamma\|x^*-y_{k+0.5}\|B_p \frac{\delta_p(x^*)\left(1+\|y_{k+0.5}-x^*\|\right)}{\sqrt{S_k}}
			\\ = ~& D_{k,p}(x^*)\|x^*-y_{k+0.5}\|+D_{k,p}(x^*)\|x^*-y_{k+0.5}\|^2.
		\end{split}
	\end{equation}
	Taking $\left(\mathbb{E}[\|\cdot\|^q|\mathcal{F}_k]\right)^{\frac{1}{q}}$ in \eqref{QQ5} and using $\left(\mathbb{E}\left[\left(\mathbb{E}[\|\cdot\|^q|\mathcal{\hat{F}}_k]\right)|\mathcal{F}_k\right]\right)^{\frac{1}{q}}=\left(\mathbb{E}[\|\cdot\|^q|\mathcal{F}_k]\right)^{\frac{1}{q}}$,
	$\left(\mathbb{E}[\|\cdot\|^q|\mathcal{\hat{F}}_k]\right)^{\frac{1}{q}}\leq\left(\mathbb{E}[\|\cdot\|^p|\mathcal{\hat{F}}_k]\right)^{\frac{1}{p}}$, we deduce that
	\begin{equation}\label{hu}
		\begin{split}
			&	\left(\mathbb{E}\left[\left(\mathbb{E}\left[\left|\langle x^*-y_{k+0.5}, t_k\hat{\epsilon}_2^k\rangle\right|^p|\mathcal{\hat{F}}_k\right]\right)^{\frac{q}{p}}|\mathcal{F}_k\right]\right)^{\frac{1}{q}}
			\\	  \geq & \left(\mathbb{E}\left[\left(\mathbb{E}\left[\left|\langle x^*-y_{k+0.5}, t_k\hat{\epsilon}_2^k\rangle\right|^q|\mathcal{\hat{F}}_k\right]\right)|\mathcal{F}_k\right]\right)^{\frac{1}{q}}
			\\  = & \left(\mathbb{E}\left[\left|\langle x^*-y_{k+0.5}, t_k\hat{\epsilon}_2^k\rangle\right|^q|\mathcal{F}_k\right]\right)^{\frac{1}{q}}.
		\end{split}
	\end{equation}
	Once again, using the properties
	$$\left(\mathbb{E}\left[\left(\mathbb{E}[\|\cdot\|^q|\mathcal{\hat{F}}_k]\right)|\mathcal{F}_k\right]\right)^{\frac{1}{q}}=\left(\mathbb{E}[\|\cdot\|^q|\mathcal{F}_k]\right)^{\frac{1}{q}},  \left(\mathbb{E}[\|\cdot\|^q|\mathcal{\hat{F}}_k]\right)^{\frac{1}{q}} \leq\left(\mathbb{E}[\|\cdot\|^p|\mathcal{\hat{F}}_k]\right)^{\frac{1}{p}},$$ and combining them with \eqref{yui}, \eqref{QQ41}, \eqref{QQ5}, \eqref{hu} and Minkowski's inequality, we deduce that
	\begin{align*}
		&	 \left(\mathbb{E}\left[\left|\langle x^*-y_{k+0.5}, t_k\hat{\epsilon}_2^k\rangle\right|^q|\mathcal{F}_k\right]\right)^{\frac{1}{q}} 
		\\	\overset{\eqref{hu}}\leq \quad&
		\left(\mathbb{E}\left[\left(\mathbb{E}\left[\left|\langle x^*-y_{k+0.5}, t_k\hat{\epsilon}_2^k\rangle\right|^p|\mathcal{\hat{F}}_k\right]\right)^{\frac{q}{p}}|\mathcal{F}_k\right]\right)^{\frac{1}{q}}
		\\ \overset{\eqref{QQ5}} \leq \quad&
		\left(\mathbb{E}\left[\left(	D_{k,p}(x^*)\|x^*-y_{k+0.5}\|+D_{k,p}(x^*)\|x^*-y_{k+0.5}\|^2\right)^q|\mathcal{F}_k\right]\right)^{\frac{1}{q}}
		\\  \leq ~\quad &D_{k,p}(x^*) \left(\mathbb{E}\left[\|x^*-y_{k+0.5}\|^q|\mathcal{F}_k\right]\right)^{\frac{1}{q}}
		+D_{k,p}(x^*)\left(\mathbb{E}\left[\|x^*-y_{k+0.5}\|^{2q}|\mathcal{F}_k\right]\right)^{\frac{1}{q}}
		\\ \leq ~\quad & D_{k,p}(x^*)\left(\mathbb{E}\left[\|x^*-y_{k+0.5}\|^p|\mathcal{F}_k\right]\right)^{\frac{1}{p}}+D_{k,p}(x^*)\left(\mathbb{E}\left[\|x^*-y_{k+0.5}\|^{p}|\mathcal{F}_k\right]\right)^{\frac{2}{p}}
		\\ \overset{\eqref{yui},\eqref{QQ41}}\leq ~&D_{k,p}(x^*)\left( I_{k,p}(x^*)\|x_k-x^*\|+D_{k,p}(x^*)\right)
		\\ &+D_{k,p}(x^*)\left(2I_{k,p}(x^*)^2\|x_k-x^*\|^2+2D_{k,p}(x^*)^2\right)
		\\ \leq ~\quad & D_{k,p}(x^*)\left( \frac{1}{2}I_{k,p}(x^*)^2+\frac{1}{2}\|x_k-x^*\|^2+D_{k,p}(x^*)\right)
		\\ &+D_{k,p}(x^*)\left(2I_{k,p}(x^*)^2\|x_k-x^*\|^2+2D_{k,p}(x^*)^2\right)
		\\ = ~\quad & \left(\frac{1}{2} D_{k,p}(x^*)+2 D_{k,p}(x^*) I_{k,p}(x^*)^2\right)\|x_k-x^*\|^2
		\\ &+\frac{1}{2} D_{k,p}(x^*) I_{k,p}(x^*)^2+ D_{k,p}(x^*)^2+2 D_{k,p}(x^*)^3.
	\end{align*}
\end{enumerate}	
\qed

The bound for the second-order moment of the oracle error sequence $\{\hat{\epsilon}_3^k\}$, which is not a martingale difference sequence, can be found in \cite{iusem2019variance}. We present this result as the following lemma.

\begin{lemma}\label{epsi3}
	\textcolor[rgb]{0.00,0.00,1.00}{{\rm\cite[Theorem 3.11]{iusem2019variance}}}
	Suppose that Assumption \ref{a22} holds. Let $\{\xi_j\}_{j=1}^N$ be an i.i.d. sample of $\xi$, and $t:\Xi\rightarrow[0, \gamma]$ be a random variable for some $\gamma\in(0,1]$. Let $z_N(x,t):=P_{X}\left(x-t\frac{1}{N}\sum_{j=1}^NT(\xi_j,x)\right)$. Then for any $p\geq2$, $x\in\mathbb{R}^n$ and $x^*\in{\rm SOL}(X, H)$, there exist positive constants $\{c_i\}_{i=1}^4$ 
	%	(depending on $n, p$ and $\bar{L}_{2p}\gamma$) 
	such that,
	$$ \left(\mathbb{E}\left[\left\|\frac{1}{N}\sum_{j=1}^N\epsilon\left(\xi_j, z_N(x, t)\right)\right\|^p\right]\right)^{\frac{1}{p}}
	\leq\dfrac{c_1\left(\mathbb{E}\left[\|\epsilon(\xi,x^*)\|^{2p}\right]\right)^{\frac{1}{2p}}+\hat{L}_{2p}\|x-x^*\|}{\sqrt{N}},$$
	where $\hat{L}_{2p}:=c_2\bar{L}_2+c_3\bar{L}_p+c_4\bar{L}_{2p}$, with $\bar{L}_p$ defined in Remark \ref{lip}.
\end{lemma}
\begin{remark}
	In \cite{iusem2019variance} , it is stated that the constants in Lemma \ref{epsi3} satisfy
	$$c_1:=2B_p+B_{2p}B_{\bar{L}\gamma,p},  \quad c_2:=\mathcal{O}\left(\left(\frac{3\sqrt{n}}{\sqrt{2}-1}+\sqrt{p}\right)B_{\bar{L}\gamma,p}\right),$$
	$$c_3:=\mathcal{O}(pB_{\bar{L}\gamma,p}), \quad c_4:=B_{2p}B_{\bar{L}\gamma,p},$$
	where $B_{\bar{L}\gamma,p} := 1 + 2\bar{L}\gamma +|L(\xi)|_{2p}\gamma$,
	with $\bar{L}$ defined in Remark~\ref{lip}.
	The constants $B_p$ and $B_{2p}$ are defined as in Lemma~\ref{QQ3}.
\end{remark}
For any $x^*\in{\rm SOL}(X, H)$, we define
\begin{equation}\label{dwa1}
	D_{p}(x^*):=\gamma B_p\delta_p(x^*) \quad\text{and} \quad I_{p}(x^*):=1+\gamma \bar{L}+D_{p}(x^*).
\end{equation}
Then, combining \eqref{dwa2} and \eqref{dwa1}, we obtain $$D_{k,p}(x^*)=\frac{1}{\sqrt{S_k}}D_{p}(x^*) \quad \text{and}\quad I_{k,p}(x^*)=1+\gamma \bar{L}+\frac{1}{\sqrt{S_k}}D_{p}(x^*).$$
Moreover, we define
\begin{equation}\label{P}
\begin{split}
	P(x^*):=&2(1-3\mu^2)(M+1)D_{2}(x^*)^2+6(M+1)D_{2}(x^*)^2(1+2D_{2}(x^*)^2)
	\\ & +6\gamma^2(M+1)c_1^2\left(\mathbb{E}\left[\|\epsilon(\xi,x^*)\|^{4}\right]\right)^{\frac{1}{2}}
	\\ & +2(M+1)\left(\frac{1}{2}D_{2}(x^*)I_{2}(x^*)^2+D_{2}(x^*)^2+2D_{2}(x^*)^3\right)
\end{split}
\end{equation}
and
\begin{equation}\label{Q}
\begin{split} Q(x^*):=&2(1-3\mu^2)(M+1)D_{2}(x^*)^2+12(M+1)D_{2}(x^*)^2I_{2}(x^*)^2
	\\ &+6\gamma^2(M+1)\hat{L}_4^2+2(M+1)\left(\frac{1}{2}D_{2}(x^*)+2D_{2}(x^*)I_{2}(x^*)^2\right).
\end{split}
\end{equation}
\begin{lemma}\label{impo}
Suppose that Assumptions \ref{Minty}, \ref{a22} and \ref{a3} hold. Let $\{x_k\}$ be the sequence generated by the Anderson(1)-SEG algorithm, and define $\iota:=\frac{\left(\min\left\{\rho\mu,\gamma\right\}\right)^2}{\mathbb{E}[L(\xi)^2]}$, where $L(\xi)\geq1$ with probability 1.
Then, for any $\hat{\beta}\in(0,\frac{1}{2}]$ and $x^*\in{\rm SOL}(X, 
H)$, we have
\begin{equation}\label{ye}
	\begin{split}
		\mathbb{E}\left[\left\|x_{k+1}-x^*\right\|^2|\mathcal{F}_k\right]\leq & \left(1+\frac{1}{\sqrt{S_k}}Q(x^*)\right) \|x_k-x^*\|^2-\frac{(1-3\mu^2)\hat{\beta}\iota}{2\gamma^2}r_{\gamma}(x_{k})^2
		\\ &+\frac{1}{\sqrt{S_k}}P(x^*)+\varrho_{k},
	\end{split}
\end{equation}
where $\varrho_k$ is defined as in Lemma \ref{lem4.66}.
\end{lemma}
	{\it \textbf{Proof}} \quad 
Taking $\mathbb{E}[\cdot|\mathcal{F}_k]$ in \eqref{lem4.6} and using the fact that $x_k\in\mathcal{F}_k$, we obtain
\begin{equation}\label{grapefruit}
	\begin{split}
		&	\mathbb{E}\left[\left\|x_{k+1}-x^*\right\|^2|\mathcal{F}_k\right]
		\\ \leq & \|x_k-x^*\|^2-\frac{(1-3\mu^2)\hat{\beta}}{2\gamma^2}r_{\gamma}(x_{k})^2\mathbb{E}\left[t_{k}^2|\mathcal{F}_k\right]
		\\ & +(1-3\mu^2)(M+1)\gamma^2\mathbb{E}\left[\left\|\hat{\epsilon}_1^{k}\right\|^2|\mathcal{F}_k\right]
		\\ & +3\gamma^2(M+1)\mathbb{E}\left[\left\|\hat{\epsilon}_2^{k}\right\|^2|\mathcal{F}_k\right]+3\gamma^2(M+1)\mathbb{E}\left[\left\|\hat{\epsilon}_3^{k}\right\|^2|\mathcal{F}_k\right]
		\\ & +2(M+1)\mathbb{E}\left[\left|\langle x^*-y_{k+0.5}, t_{k}\hat{\epsilon}_2^{k}\rangle\right||\mathcal{F}_k\right]+\varrho_{k}.
	\end{split}
\end{equation}
%\begin{enumerate}[{\rm (i)}]
%	\item From Lemma \ref{QQ4}, we have
%	\begin{equation}\label{tab1}
	%		\mathbb{E}[\|\hat{\epsilon}_1^k\|^2|\mathcal{F}_k]	\leq2B^2_2\delta_2(x^*)^2\frac{\left(1+\|x_k-x^*\|^2\right)}{S_k},
	%	\end{equation}
%	\begin{equation}\label{tab2}
	%		\mathbb{E}[\|\hat{\epsilon}_2^{k}\|^2|\mathcal{F}_k]\leq\frac{2}{S_k}B_2^2\delta_2(x^*)^2\left(1+2I_{k,2}(x^*)^2\|x_k-x^*\|^2+2D_{k,2}(x^*)^2\right)
	%	\end{equation}
%	and
%	\begin{equation}\label{tab3}
	%		\begin{split}
		%			&	\mathbb{E}\left[\left|\langle x^*-y_{k+0.5}, t_{k}\hat{\epsilon}_2^{k}\rangle\right||\mathcal{F}_k\right]
		%			\\ \leq & \left(\frac{1}{2} D_{k,2}(x^*)+2 D_{k,2}(x^*) I_{k,2}(x^*)^2\right)\|x_k-x^*\|^2
		%			\\ &+\frac{1}{2} D_{k,2}(x^*) I_{k,2}(x^*)^2+ D_{k,2}(x^*)^2+2D_{k,2}(x^*)^3.
		%		\end{split}
	%	\end{equation}
%	\item From Lemma \ref{epsi3} and the fact that $\xi^k$ is independent of $\mathcal{F}_k$, it follows that
%	\begin{equation}\label{tab4}
	%		\begin{split}
		%			\mathbb{E}[\|\hat{\epsilon}_3^{k}\|^2|\mathcal{F}_k]=&\mathbb{E}[\|\hat{\epsilon}(\xi^k, y(\xi^k, x_k, t_k))\|^2]
		%			\\ \leq & \frac{2c_1^2\left(\mathbb{E}\left[\|\epsilon(\xi, x^*)\|^{4}\right]\right)^{\frac{1}{2}}+2\hat{L}_4^2\|x_k-x^*\|^2}{S_k}.
		%		\end{split}
	%	\end{equation}
%	From Lemma \ref{tab222}, we get
%	\begin{equation}\label{tab5}
	%		\mathbb{E}[t_{k}^2|\mathcal{F}_k]\geq \frac{\left(\min\left\{\rho\mu,\gamma\right\}\right)^2}{\mathbb{E}[L(\xi)^2]}=\frac{\left(\min\left\{\rho\mu,\gamma\right\}\right)^2}{|L(\xi)|^2_2}=\iota.
	%	\end{equation}
%	\end{enumerate}	
From Lemma \ref{tab222}, we get
\begin{equation*}
\mathbb{E}[t_{k}^2|\mathcal{F}_k]\geq \frac{\left(\min\left\{\rho\mu,\gamma\right\}\right)^2}{\mathbb{E}[L(\xi)^2]}=\iota.
\end{equation*}		
Moreover, since $\xi^{k}$ is independent of $\mathcal{F}_k$, we have
$$\mathbb{E}\left[\left\|\hat{\epsilon}_3^{k}\right\|^2|\mathcal{F}_k\right]=\mathbb{E}\left[\left\|\hat{\epsilon}_3^{k}\right\|^2\right]=\mathbb{E}\left[\left\|\frac{1}{S_k}\sum_{j=1}^{S_k}\epsilon\left(\xi_j, z_{S_k}(x_k, t_k)\right)\right\|^2\right].$$
Combining these results with Lemmas \ref{QQ4} and \ref{epsi3}, we can bound the terms in \eqref{grapefruit} separately. Consequently, we obtain
\begin{equation*}
\begin{split}
	\mathbb{E}\left[\left\|x_{k+1}-x^*\right\|^2|\mathcal{F}_k\right]\leq & \left(1+\frac{1}{\sqrt{S_k}}Q(x^*)\right) \|x_k-x^*\|^2-\frac{(1-3\mu^2)\hat{\beta}\iota}{2\gamma^2}r_{\gamma}(x_{k})^2
	\\ &+\frac{1}{\sqrt{S_k}}P(x^*)+\varrho_{k}.
\end{split}
\end{equation*}
\qed

Based on the above lemmas, %and the following Egorov's theorem
we present the main conclusion.
% {\color{red}Don't find the reference.}
%\begin{theorem}\label{Egorov}[Egorov's Theorem]\cite{gagaeff1932suites}
%	Let $(X, \mathcal{F}, \kappa)$ be a space with a finite nonnegative measure $\varkappa$ and let $\varkappa$-measurable functions $f_n$ be such that $\varkappa$-almost everywhere there is a finite limit $f(x):=\lim_{n\rightarrow\infty}f_n(x).$ Then, for every $\zeta>0$, there exists a set $X_{\zeta}\in\mathcal{F}$ such that $\varkappa(X\backslash X_{\zeta})<\zeta$ and the functions $f_n$ converge to $f$ uniformly on $X_{\zeta}$.
%\end{theorem}
\begin{theorem}
Suppose that Assumptions \ref{Minty}, \ref{a22} and \ref{a3} hold. 
Let $\{x_k\}$ be the sequence generated by the Anderson(1)-SEG algorithm. Then, a.s. the sequence $\{x_k\}$ converges to a solution of ${\rm SVI}(X, H)$ and $r_{\gamma}(x_k)$ converges to 0. 
\end{theorem}
{\it \textbf{Proof}}\quad 
Let $x^*$ be a solution of ${\rm SVI}(X, H)$.
Since $\sum_{k=0}^{\infty}\frac{1}{\sqrt{S_k}}<\infty$ (Assumption \ref{a3}) and $\sum_{k=0}^\infty \varrho_k<\infty$ (Lemma \ref{lem4.66} and $\tau>\frac{1}{2}$), it follows that $$\sum_{k=0}^{\infty}\frac{1}{\sqrt{S_k}}Q(x^*)<\infty \quad{\rm and} \quad \sum_{k=0}^{\infty}\left(\frac{1}{\sqrt{S_k}}P(x^*)+\varrho_{k}\right)<\infty.$$ Recalling Lemma \ref{impo}, and using the fact that $\mu\in(0,\frac{\sqrt{3}}{3})$, we apply Theorem \ref{a.s.convergence} with $v_{k}=\|x_k-x^*\|^2,$ $a_k=\frac{1}{\sqrt{S_k}}Q(x^*),$ $u_k=\frac{(1-3\mu^2)\hat{\beta}\iota}{2\gamma^2}r_{\gamma}(x_{k})^2$ and $b_k=\frac{1}{\sqrt{S_k}}P(x^*)+\varrho_{k}.$ This yields that a.s. $\|x_k-x^*\|^2$ converges and $\sum_{k=0}^{\infty}r_{\gamma}(x_{k})^2<\infty$.

%By Theorem \ref{Egorov}, we deduce that for any $\zeta\in(0,1)$, there exists a measurable set $X_{\zeta}\subseteq X$ such that  $\mathbb{P}(X_{\zeta})\geq 1-\zeta$ and $\|x_k-x^*\|^2$ converges uniformly on $X_{\zeta}$.
In particular, a.s. $\{x_k\}$ is bounded, which implies that, a.s., there exists a cluster point $\hat{x}$ and a subsequence $\{x_{s_j}\}$ satisfying 
\begin{equation}\label{subsequence}
\lim_{j\rightarrow\infty}x_{s_j}=\hat{x}.
\end{equation}
Moreover, we have that a.s.
\begin{equation}\label{xuan} \lim_{k\rightarrow\infty}r_{\gamma}(x_k)^2=\lim_{k\rightarrow\infty}\|x_k-P_{X}\left(x_k-\gamma H( x_k)\right)\|^2=0.
\end{equation}
From \eqref{xuan} and the continuity of both $H$ and $P_{X}$, we deduce that a.s.
$\|\hat{x}-P_{X}(\hat{x}-\gamma H(\hat{x}))\|=0.$ Together with Lemma \ref{equ}, this implies that a.s. $\hat{x}\in{\rm SOL}(X, H)$. 
Since a.s. $\|x_k-x^*\|$ converges for any $x^*\in{\rm SOL}(X, H)$, it follows that a.s. $\|x_k-\hat{x}\|$ also converges. Together with \eqref{subsequence}, we have that a.s.
$$\lim_{k\rightarrow\infty}\|x_k-\hat{x}\|=\lim_{j\rightarrow\infty}\|x_{s_j}-\hat{x}\|=0.$$
\qed

Before presenting the conclusion about the convergence rate, we first provide an upper bound for $\sup_{k\geq k_0}\mathbb{E}\left[\left\|x_k-x^*\right\|^2\right]$.
\begin{lemma}\label{boundsup}
Suppose Assumptions \ref{Minty}, \ref{a22} and \ref{a3} hold. Let $\{x_k\}$ be the sequence generated by the Anderson(1)-SEG algorithm. Let $x^*\in{\rm SOL}(X, H)$, and choose $k_0\in\mathbb{N}$, $\phi\in(0,1)$ and $\psi>0$ such that	$\sum_{k=k_0}^{\infty}\frac{1}{\sqrt{S_k}}<\frac{\phi}{Q(x^*)}$ and $\sum_{k=k_0}^{\infty}\varrho_{k}<\psi$. Then, we have$$\sup_{k\geq k_0}\mathbb{E}\left[\left\|x_k-x^*\right\|^2\right]\leq \frac{\mathbb{E}\left[\left\|x_{k_0}-x^*\right\|^2\right]+P(x^*)\frac{\phi}{Q(x^*)}+\psi}{1-\phi}.$$
\end{lemma}
{\it \textbf{Proof}}\quad Taking expectation in \eqref{ye} and applying $\mathbb{E}\left[\mathbb{E}\left[\cdot|\mathcal{F}_k\right]\right]=\mathbb{E}[\cdot]$, we drop the negative term to obtain
\begin{equation}\label{eqs}
\begin{split}
	\mathbb{E}\left[\left\|x_{k+1}-x^*\right\|^2\right]\leq \left(1+\frac{1}{\sqrt{S_k}}Q(x^*)\right) \mathbb{E}\left[\|x_k-x^*\|^2\right]+\frac{1}{\sqrt{S_k}}P(x^*)+\varrho_{k}.
\end{split}
\end{equation}
Let $d_i=\|x_i-x^*\|$ for $i\in\mathbb{N}_0$ and $k> k_0$.
Then, summing \eqref{eqs} from $k_0$ to $k-1$ yields 
\begin{equation}\label{dk}
\mathbb{E}\left[d_k^2\right]\leq \mathbb{E}\left[d_{k_0}^2\right] +Q(x^*)\sum_{i=k_0}^{k-1}\frac{1}{\sqrt{S_i}}\mathbb{E}\left[d_i^2\right]+P(x^*)\sum_{i=k_0}^{k-1}\frac{1}{\sqrt{S_i}}+\sum_{i=k_0}^{k-1}\varrho_{i}.
\end{equation}
For any $a>0$, we define $\varpi_a:=\inf\{k>k_0: \mathbb{E}\left[d_k^2\right]^{\frac{1}{2}}> a\}$. Assume that $\mathbb{E}\left[d_k^2\right]$ is unbounded, then $\varpi_a<\infty$ for all $a>0$. By choosing $k=\varpi_a$ in \eqref{dk} and considering that 
$\sum_{k=k_0}^{\infty}\frac{1}{\sqrt{S_k}}<\frac{\phi}{Q(x^*)}$ and $\sum_{k=k_0}^{\infty}\varrho_{k}<\psi$, we obtain
\begin{equation*}
\begin{split}
	a^2<\mathbb{E}\left[d_{\varpi_a}^2\right]\leq & \mathbb{E}\left[d_{k_0}^2\right]+Q(x^*)\sum_{i=k_0}^{\varpi_a-1}\frac{1}{\sqrt{S_i}}\mathbb{E}\left[d_{i}^2\right]+P(x^*)\sum_{i=k_0}^{\varpi_a-1}\frac{1}{\sqrt{S_i}}+\sum_{i=k_0}^{\varpi_a-1}\varrho_{i}
	%\\ \leq & \left|d_{k_0}\right|_2^2+Q(x^*)\sum_{i=k_0}^{\varpi_a-1}\frac{1}{\sqrt{S_i}}a^2+P(x^*)\sum_{i=k_0}^{\varpi_a-1}\frac{1}{\sqrt{S_i}}+\sum_{i=k_0}^{k-1}\varrho_{k}
	\\ \leq &\mathbb{E}\left[d_{k_0}^2\right]+\phi a^2+P(x^*)\frac{\phi}{Q(x^*)}+\psi,
\end{split}
\end{equation*}
which implies
$$a^2\leq \frac{\mathbb{E}\left[d_{k_0}^2\right]+P(x^*)\frac{\phi}{Q(x^*)}+\psi}{1-\phi}.$$This contradicts the arbitrariness of $a$, thus $ \mathbb{E}\left[d_{k}^2\right]$ is bounded and 
$$\sup_{k\geq k_0}\mathbb{E}\left[d_{k}^2\right]\leq \frac{\mathbb{E}\left[d_{k_0}^2\right]+P(x^*)\frac{\phi}{Q(x^*)}+\psi}{1-\phi}.$$
\qed
\begin{remark} 
Let $x^*\in{\rm SOL}(X, H)$. From Lemma \ref{boundsup}, it follows that under Assumptions \ref{Minty}, \ref{a22} and \ref{a3}, there exist 
$k_0\in\mathbb{N}$, $\phi\in(0,1)$ and $\psi>0$ such that
\begin{equation}\label{boundee} 
	\sup_{k\geq 0}\mathbb{E}\left[\|x_k-x^*\|^2\right]\leq \frac{\max_{k=0,\cdots,k_0}\mathbb{E}\left[\left\|x_{k_0}-x^*\right\|^2\right]+P(x^*)\frac{\phi}{Q(x^*)}+\psi}{1-\phi}.
\end{equation}
\end{remark}
Based on the above results, the following two theorems establish a sublinear convergence rate for the mean residual function, as well as the optimal oracle complexity of the proposed algorithm.
\begin{theorem}\label{rates}
Suppose that Assumptions \ref{Minty}, \ref{a22} and \ref{a3} hold. Let $\{x_k\}$ be the sequence generated by the Anderson(1)-SEG algorithm. Then we have
$$\min_{0\leq k \leq N}\mathbb{E}\left[r_{\gamma}(x_k)^2\right]=\mathcal{O}\left(\frac{1}{N}\right).$$
\end{theorem}
{\it \textbf{Proof}}\quad  Let $x^*\in{\rm SOL}(X, H)$. Taking expectation in \eqref{ye} and using $\mathbb{E}\left[\mathbb{E}\left[\cdot|\mathcal{F}_k\right]\right]=\mathbb{E}[\cdot]$, we recursively sum the resulting inequality from $k=0$ to $k=N$ to obtain
\begin{equation*}
\begin{split}
	\frac{(1-3\mu^2)\hat{\beta}\iota}{2\gamma^2}\sum_{k=0}^N\mathbb{E}\left[r_{\gamma}(x_{k})^2\right]\leq & \|x_0-x^*\|^2+Q(x^*)\sum_{k=0}^{N}\frac{1}{\sqrt{S_k}}\mathbb{E}\left[\|x_k-x^*\|^2\right]
	\\ &+P(x^*)\sum_{k=0}^{N}\frac{1}{\sqrt{S_k}}+\sum_{k=0}^{N}\varrho_{k}.
\end{split}
\end{equation*}
Since $\sum_{k=0}^{\infty}\frac{1}{\sqrt{S_k}}<\infty$ (Assumption \ref{a3}), $\sum_{k=0}^\infty \varrho_k<\infty$ (Lemma \ref{lem4.66} and $\tau>\frac{1}{2}$) and \eqref{boundee}, we know that there exists $\Phi>0$ such that
\begin{equation*}
\begin{split}
	\sum_{k=0}^N\mathbb{E}\left[r_{\gamma}(x_{k})^2\right]\leq \frac{2\gamma^2}{(1-3\mu^2)\hat{\beta}\iota}\Phi,
\end{split}
\end{equation*}
which implies
$$\min_{0\leq k\leq N}\mathbb{E}\left[r_{\gamma}(x_{k})^2\right]\leq\frac{1}{N+1}\sum_{k=0}^N\mathbb{E}\left[r_{\gamma}(x_{k})^2\right] \leq\frac{2\gamma^2}{(N+1)(1-3\mu^2)\hat{\beta}\iota}\Phi.$$
\qed
\begin{theorem}
Suppose that Assumptions \ref{Minty}, \ref{a22} and \ref{a3} hold. Let $\{x_k\}$ be the sequence generated by the Anderson(1)-SEG algorithm. Set $S_k:=\mathcal{N}\lceil (k+\lambda)^2(\ln (k+\lambda))^{2+2b} \rceil$ for $\mathcal{N}=\mathcal{O}(n^2), b>0$ and $\lambda>2.$ Given $\varsigma>0$, the Anderson(1)-SEG algorithm achieves the tolerance
$$ \min_{0\leq k\leq N}\mathbb{E}\left[r_{\gamma}(x_{k})^2\right]\ \leq \varsigma$$
after $N=\mathcal{O}(b^{-1}\varsigma^{-1})$ iterations and the oracle complexity $	\sum_{k=0}^{N}(1+m_k)S_k$ is bounded above by
$$b^{-3}\log_{\frac{1}{\rho}}\left(\frac{\gamma\tilde{L}_k}{\min\{\rho\mu,\gamma\}}\right)\ln (b^{-1}\varsigma^{-1})^{2+2b}\mathcal{O}(n^2\varsigma^{-3})$$
with probability 1, where $m_k$ is the number of oracle calls used in the line search scheme \eqref{lins} at the $k$-iteration and $\tilde{L}_k$ is defined as in Lemma \ref{tab222}.
\end{theorem}
{\it \textbf{Proof}} \quad 
From the form of $S_k$, we can compute that 
\begin{equation*}
\sum_{k=0}^{\infty}\frac{1}{\sqrt{S_k}}	\leq \int_{-1}^{\infty} \frac{dt}{\sqrt{\mathcal{N}}(t+\lambda)\left(\ln(t+\lambda)\right)^{1+b}}=\frac{1}{\sqrt{\mathcal{N}}b\left(\ln(\lambda-1)\right)^b}<\infty.
\end{equation*} 
According to the definitions of $\{c_i\}_{i=1}^4$, $\hat{L}_{2p}$ in Lemma \ref{epsi3}, $P(x^*)$ in \eqref{P} and $Q(x^*)$ in \eqref{Q}, along with Theorem \ref{rates}, there exists a constant $U>0$ such that for any $k\in\mathbb{N}_0$, $\min_{0\leq k\leq N}\mathbb{E}\left[r_{\gamma}(x_{k})^2\right]\leq nU(\mathcal{\sqrt{N}}bN)^{-1}.$ Hence, given $\varsigma>0$, we obtain $$ \min_{0\leq k\leq N}\mathbb{E}\left[r_{\gamma}(x_{k})^2\right]\ \leq \varsigma$$ after $N=\mathcal{O}(n(\mathcal{\sqrt{N}}b\varsigma)^{-1})$ iterations.

Let $m_k$ be the number of line search in the iteration $k$. After $N$ iterations, the oracle complexity is upper bounded by
\begin{equation}\label{ql}
\begin{split}
	&\sum_{k=0}^{N}(1+m_k)S_k
	\\ \leq & \left(1+\max_{k\in\{0,\cdots,N\}}m_k\right)\sum_{k=0}^{N}\left({\mathcal{N} (k+\lambda)^2(\ln (k+\lambda))^{2+2b}+1}\right)
	\\ = & \left(1+\max_{k\in\{0,\cdots,N\}}m_k\right)\sum_{k=\lambda}^{N+\lambda}\left({\mathcal{N} k^2(\ln k)^{2+2b}+1}\right)
	\\ \leq & \left(1+\max_{k\in\{0,\cdots,N\}}m_k\right)\left({\mathcal{N} (N+1)(N+\lambda)^2(\ln (N+\lambda))^{2+2b}+1}\right)
	\\ \leq & \left(1+\max_{k\in\{0,\cdots,N\}}m_k\right)\left(\mathcal{N}{ (b^{-1}\varsigma^{-1}+1)(b^{-1}\varsigma^{-1}+\lambda)^2(\ln (b^{-1}\varsigma^{-1}+\lambda))^{2+2b}+1}\right).
\end{split}
\end{equation}
Moreover, Lemma \ref{tab222} implies that a.s. $m_k\leq \log_{\frac{1}{\rho}}\left(\frac{\gamma\tilde{L}_k}{\min\{\rho\mu,\gamma\}}\right)$. 
Combining this with \eqref{ql} and $\mathcal{N}=\mathcal{O}(n^2)$ yields the claimed bound on $\sum_{k=0}^{N}(1+m_k)S_k$.
\qed
\section{Numerical experiments}
\label{section4}
In this section, we conduct numerical examples to compare the Anderson(1)-SEG algorithm with the SEG algorithm that incorporates line search \eqref{lins}. In addition, we also perform numerical experiments on portfolio selection problems using real data.
All the codes are written in Matlab (R2023b) and run on a MacBook Air (16.00GB of RAM).
	\subsection{Simple examples}
In the following experiments of this subsection,  
the stopping criterion for the two examples is set as
$$\tilde{r}_{\xi^k}(x_k):=\left\|F_{1,\xi^k}\right\|=\left\|P_{X}\left(x_k- H_{\xi^k}( x_k)\right)-x_k\right\|<10^{-2}$$
or when the number of iterations exceeds $200.$
The parameters for the Anderson(1)-SEG algorithm are specified as follows:
\begin{equation}\label{para}
	\nu=30, ~M=5000, ~\tau=0.6, ~\rho=0.8, ~\mu=0.5.
\end{equation}
In the figures and tables presented in this subsection, ‘Sec.’ denotes the CPU time measured in seconds, while ‘Iter.’ indicates the total number of iterations. 
The numbers in parentheses denote the number of executed Anderson steps \eqref{SAnderson}.
Additionally, the algorithm that performs best regarding the average number of iterations and CPU time is emphasized in bold for each combination of dimension $n$ and parameter $\gamma$.
\begin{example}\cite{cai20141}\label{exam41}
	Consider the following stochastic complementarity problem
	\begin{equation}\label{sc}
		\mathbb{E}[T(\xi,x)]\geq \mathbf{0}, ~~~x\geq\mathbf{0}, ~~~x^{{\rm T}}\mathbb{E}[T(\xi,x)]=0,
	\end{equation}
	where $T(\xi,x):=D(\xi, x)+Q(\xi)x+q(\xi)$ with $D(\xi,x)\in \mathbb{R}^n$ and $Q(\xi)x+q(\xi)\in \mathbb{R}^n$ to be its nonlinear and linear parts, respectively. Assume that
	\begin{enumerate}[{\rm (i)}]
		\item for $i = 1, \ldots, n$, the $i$-th component of $D(\xi,x)$ is given by $D_i(\xi,x):=d_i(\xi)\cdot{\rm arctan}(a_i(\xi)x_i),$ where $d(\xi)$ and $a(\xi)$ are random vectors in $\mathbb{R}^n$ with elements obeying uniform distributions on (0, 1);
		\item $Q(\xi):=B+Y(\xi),$ where $B\in\mathbb{R}^{n\times n}$ is a deterministic skew symmetric matrix with elements obeying uniform distributions on (0, 5) and $Y(\xi)\in\mathbb{R}^{n\times n}$ is a random diagonal matrix with elements obeying uniform distributions on (0, 2);
		\item $q(\xi)\in\mathbb{R}^n$ is a random vector with elements obeying uniform distributions on (-3, 3).
	\end{enumerate}
\end{example}
From the above generation method, it can be seen that $H(\cdot):=	\mathbb{E}[T(\xi,\cdot)]$ is monotone. Moreover, the stochastic complementarity problem given in \eqref{sc} can be reformulated as ${\rm SVI}(X, H)$ by choosing $X=\mathbb{R}_+^n$.

\begin{figure}[H]
	\centering
	\subfigure[$n=50$, $\gamma=0.05$]{\includegraphics[width=0.48\textwidth]{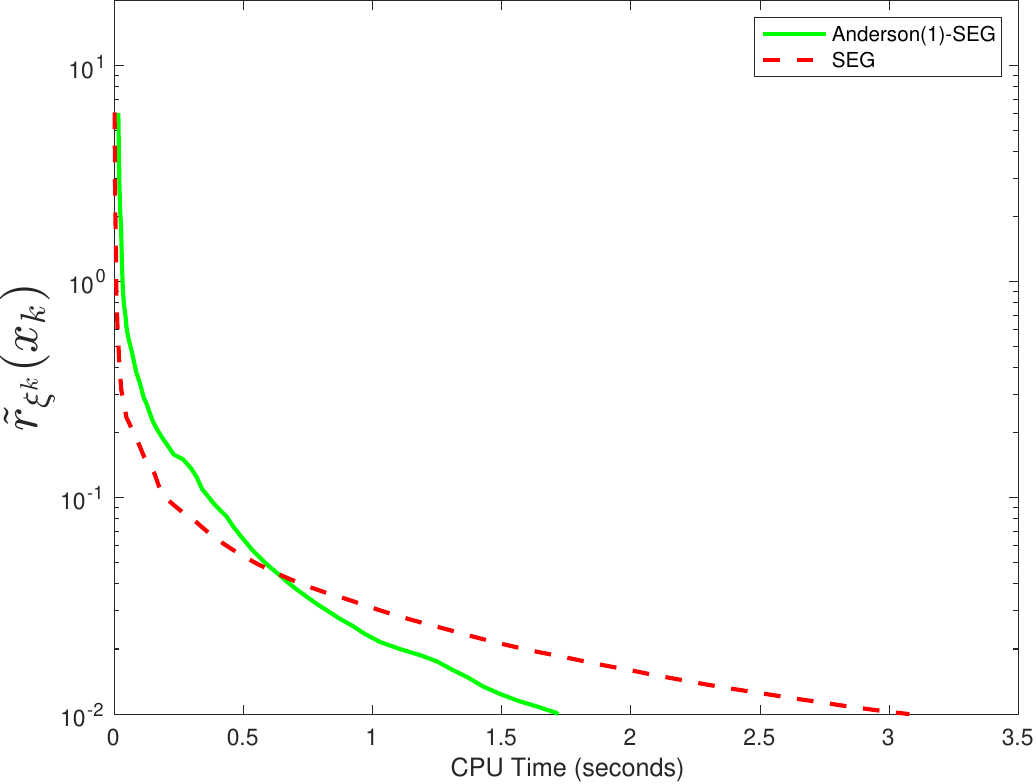}}
	\hfill
	\subfigure[$n=70$, $\gamma=0.03$]{\includegraphics[width=0.48\textwidth]{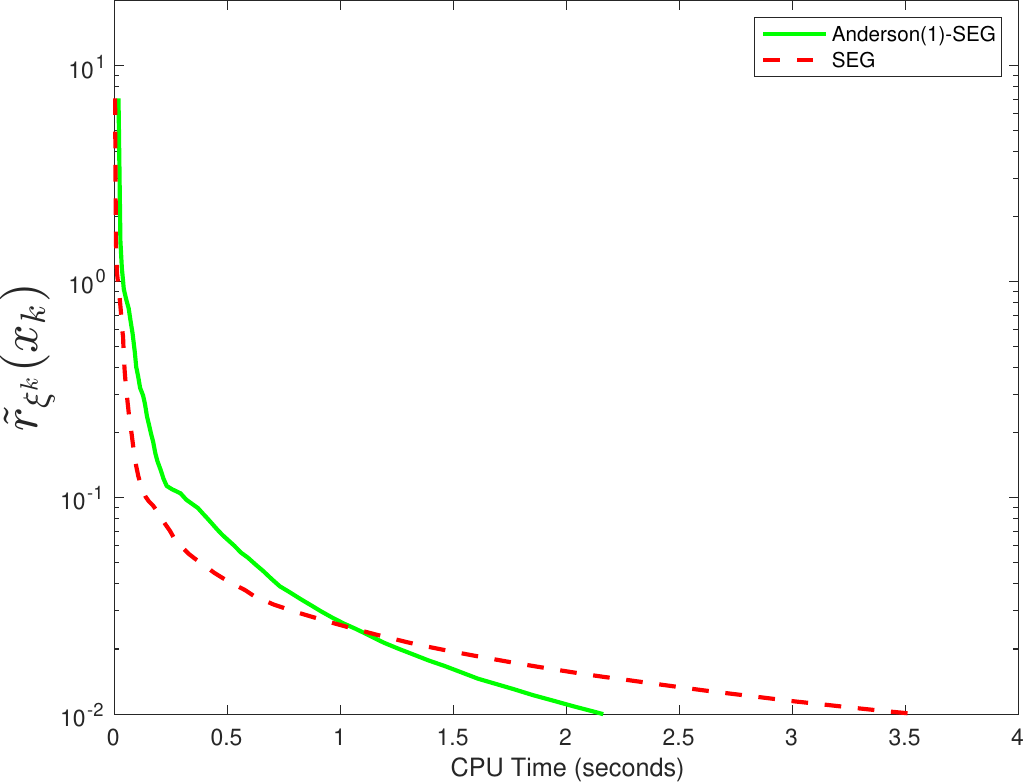}}
	
	\vspace{1em}
	
	\subfigure[$n=100$, $\gamma=0.02$]{\includegraphics[width=0.48\textwidth]{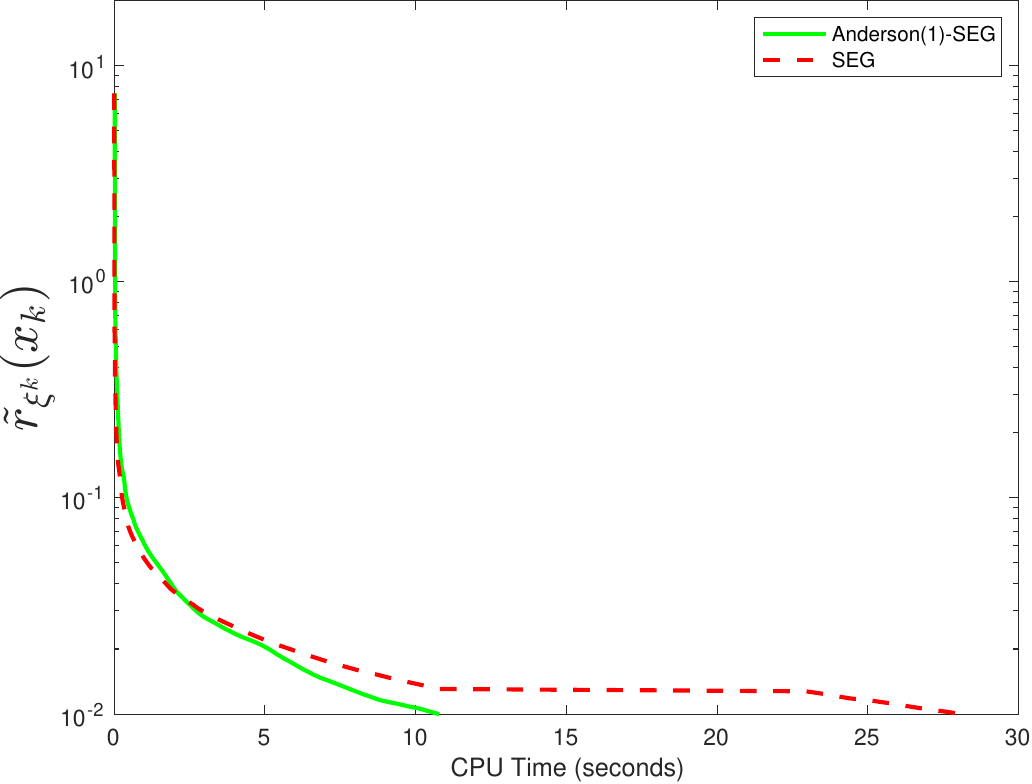}}
	\hfill
	\subfigure[$n=150$, $\gamma=0.01$]{\includegraphics[width=0.48\textwidth]{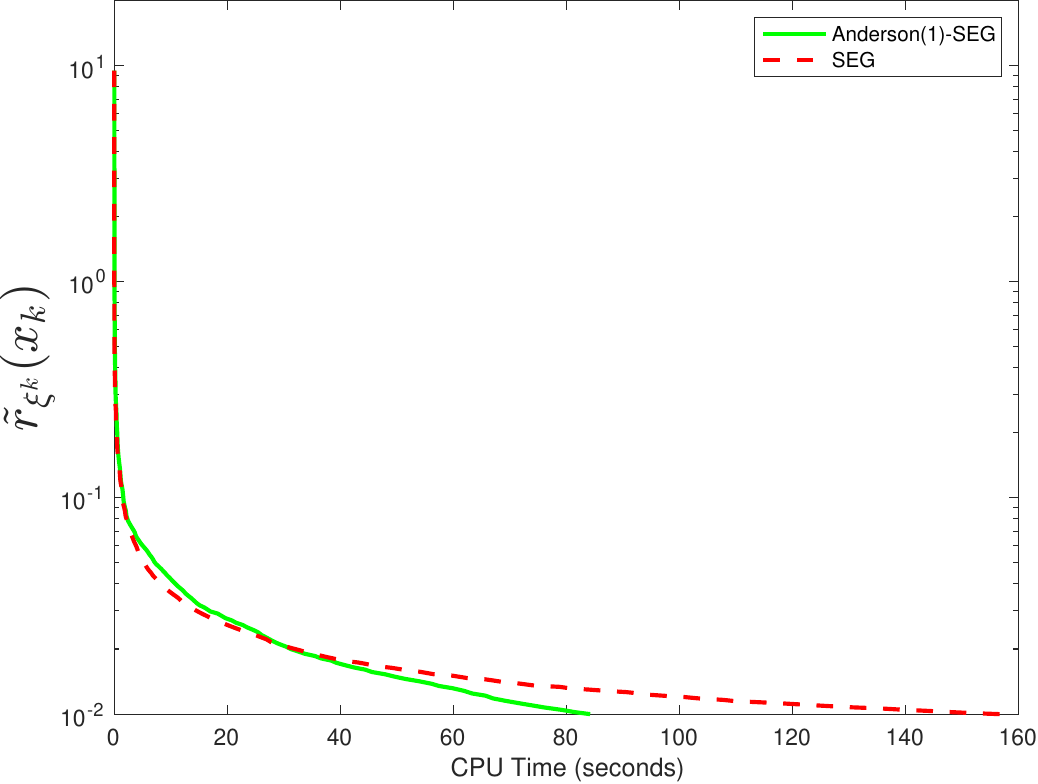}}
	
	\caption{Comparisons of the convergence behaviours of the Anderson(1)-SEG and the SEG algorithms for Example \ref{exam41}}
	\label{fig10203050}
\end{figure}
Let $S_k:=\mathcal{N}\lceil (k+\lambda)^2(\ln (k+\lambda))^{2+2b} \rceil$ with $\lambda=2+10^{-4}$, $b=10^{-4}$ and $\mathcal{N}=2$. 
Using the Anderson(1)-SEG and the SEG algorithms
to solve Example~\ref{exam41}, where both algorithms are initialized
from the same randomly chosen initial point, the experimental
results are presented in Figure~\ref{fig10203050}.
This figure depicts the variation of $\tilde{r}_{\xi^k}(x_k)$ with respect to CPU time for dimensions $n = 50, 70, 100$, and $150$. It can be seen that, upon reaching the stopping criterion, the Anderson(1)-SEG algorithm requires less CPU time than that of the SEG algorithm.

We also conduct experiments with $S_k$ set to $\mathcal{N}\lceil (k+\lambda)(\ln (k+\lambda))^{1+b} \rceil$, where $\lambda=5$, $b=0.1$ and $\mathcal{N}=20$. The corresponding results are shown in Table  \ref{t411}. 
Table \ref{t411} presents the average number of iterations and average CPU time for two algorithms across different dimensions based on ten independent experiments. 
As can be seen from Table \ref{t411}, 
in most cases, its numerical performance, in terms of both the number of iterations and CPU time, is better than that of the SEG algorithm.
\begin{table}[H]
	\centering
	\label{table2}
		\caption{Comparisons of the two algorithms for Example \ref{exam41}}\label{t411}
	\begin{tabular}{c|cccc}
		\toprule
		Sec.(avr)\\ Iter.(avr) &  & Anderson(1)-SEG & SEG \\
		\midrule 
		\multirow{2}{1cm}{$n=10$ $\gamma=0.1$} &  &    0.0407  &     \textbf{0.0355}  \\
		&  & \textbf{36.7} (27.7)   &  40   \\
		\midrule
		\multirow{2}{1cm}{$n=20$ $\gamma=0.08$} &  & \textbf{0.1376}  &   0.1390  \\
		& & \textbf{ 62.3} (41.9)  &  67.5 \\
		\midrule
		\multirow{2}{1cm}{$n=30$ $\gamma=0.03$} &  & \textbf{0.1258}  &   0.1372  \\
		& & \textbf{59.1} (46.1)   &    67.5 \\
		\midrule
		\multirow{2}{1cm}{$n=50$ $\gamma=0.02$} & & \textbf{0.3442
		} &   0.4192\\
		& & \textbf{80} (54.1)   &    94\\
		\midrule
		\multirow{2}{1cm}{$n=100$ $\gamma=0.02$} &  &  \textbf{0.7064}  &  0.8094\\
		& & \textbf{94.9} (63.8)   & 106.3 \\
		\bottomrule
	\end{tabular}
\end{table}
	\begin{example}\cite{kannan2019optimal}\label{exam42}
	Consider the following stochastic fractional convex quadratic problem
	\begin{equation}\label{ex442}
		\min_{x\in X}\mathbb{E}\left[\frac{f(\xi,x)}{g(x)}\right],
	\end{equation}
	where $$f(\xi,x):=0.5x^{\rm T}\left(0.025U^{\rm T}U+0.025\frac{\|U^{\rm T}U\|_F}{\|V(\xi)\|_F}V(\xi)\right)x+0.5\left((c+\bar{c}(\xi))^{\rm T}x+4n\right)^2,$$ $g(x):=r^{\rm T}x+d+4n$, and $X:=\{x:\mathbf{0}\leq x \leq 4\mathbf{e}\}$.
\end{example}
We note that \(V(\xi) \in \mathbb{R}^{n \times n}\) and \(\bar{c}(\xi) \in \mathbb{R}^n\) are randomly generated, where the elements of the matrix \(V(\xi)\) are drawn from a standard normal distribution and the elements of the vector \(\bar{c}(\xi)\) are drawn from a uniform distribution on (0, 1). 
Furthermore, \(U \in \mathbb{R}^{n \times n}\) and \(c \in \mathbb{R}^n\) are deterministic matrix and vector, respectively, whose elements are generated once from the standard normal distribution, while the elements of the vector \(r \in \mathbb{R}^n\) and the scalar \(d \in \mathbb{R}\) are generated once from uniform distributions on (0, 5).

It is easy to see that problem \eqref{ex442} is equivalent to the ${\rm SVI}(X, H)$ with
\begin{equation*}
	\begin{split}
		H(x)=\mathbb{E}[T(\xi, x)]= &\nabla_x \mathbb{E}\left[\frac{f(\xi,x)}{g(x)}\right] = \mathbb{E}\left[\frac{\nabla_x f(\xi,x) \cdot g(x) - f(\xi,x) \cdot \nabla g(x)}{[g(x)]^2}\right].
	\end{split}
\end{equation*}
By \cite{kannan2019optimal}, we know that $H$ is pseudomonotone on $X$.

Let $\gamma = 0.9$ and $S_k:=\mathcal{N}\lceil (k+\lambda)^2(\ln (k+\lambda))^{2+2b} \rceil$ with $\lambda=2+10^{-4}$, $b=10^{-4}$ and $\mathcal{N}=2$. For each dimension, both algorithms start
from the same randomly chosen initial point.
Figure~\ref{fig1020} presents the variation of $\tilde{r}_{\xi^k}(x_k)$ with
respect to CPU time for different dimensions. It can be observed that the numerical performance of the Anderson(1)-SEG algorithm is superior to that of the SEG algorithm for this problem.
\begin{figure}[H]
	\centering
	\subfigure[$n=10$]{\includegraphics[width=0.48\textwidth]{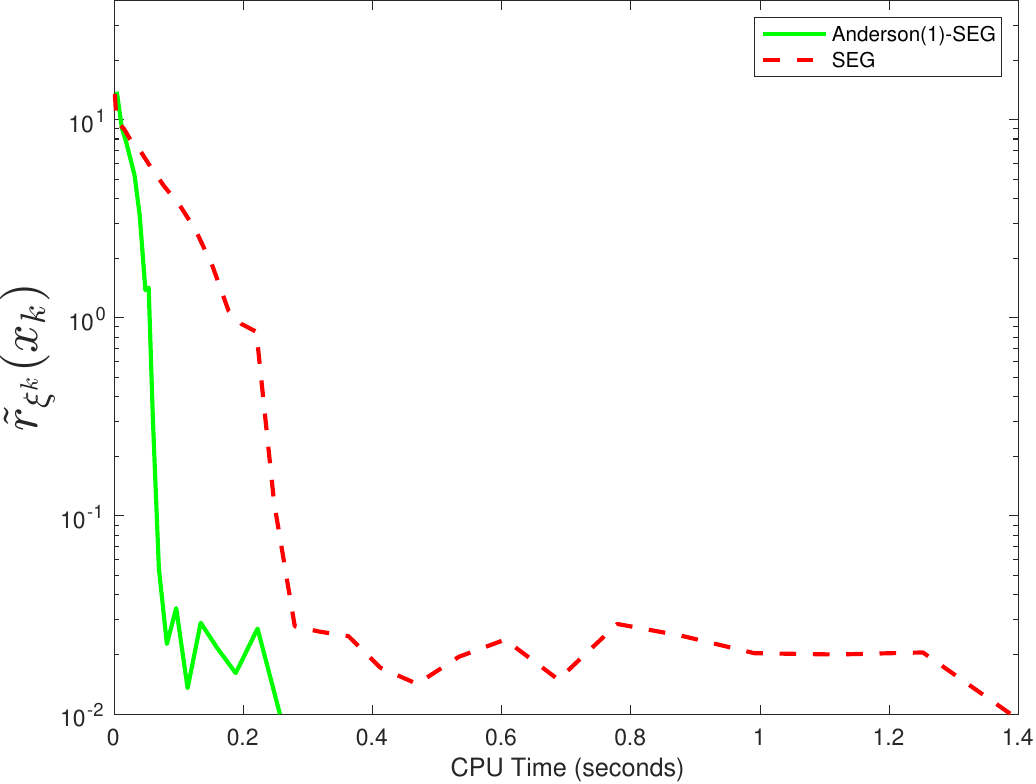}}
	\hfill
	\subfigure[$n=20$]{\includegraphics[width=0.48\textwidth]{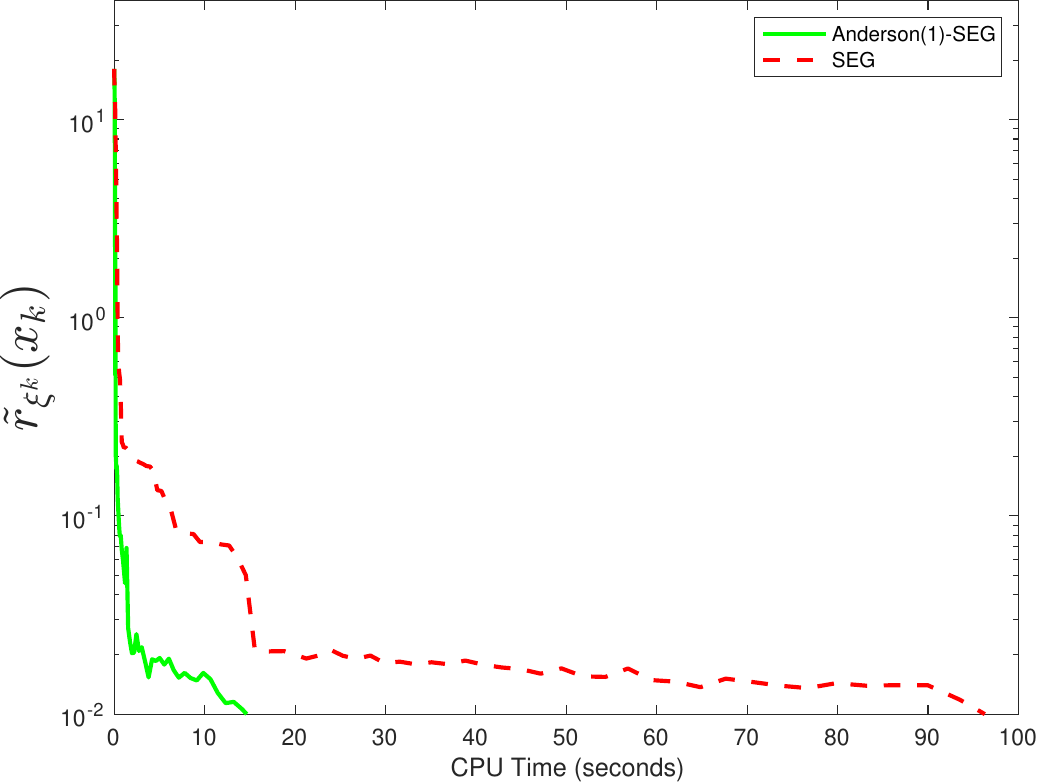}}
	
	\vspace{1em}
	
	\subfigure[$n=30$]{\includegraphics[width=0.48\textwidth]{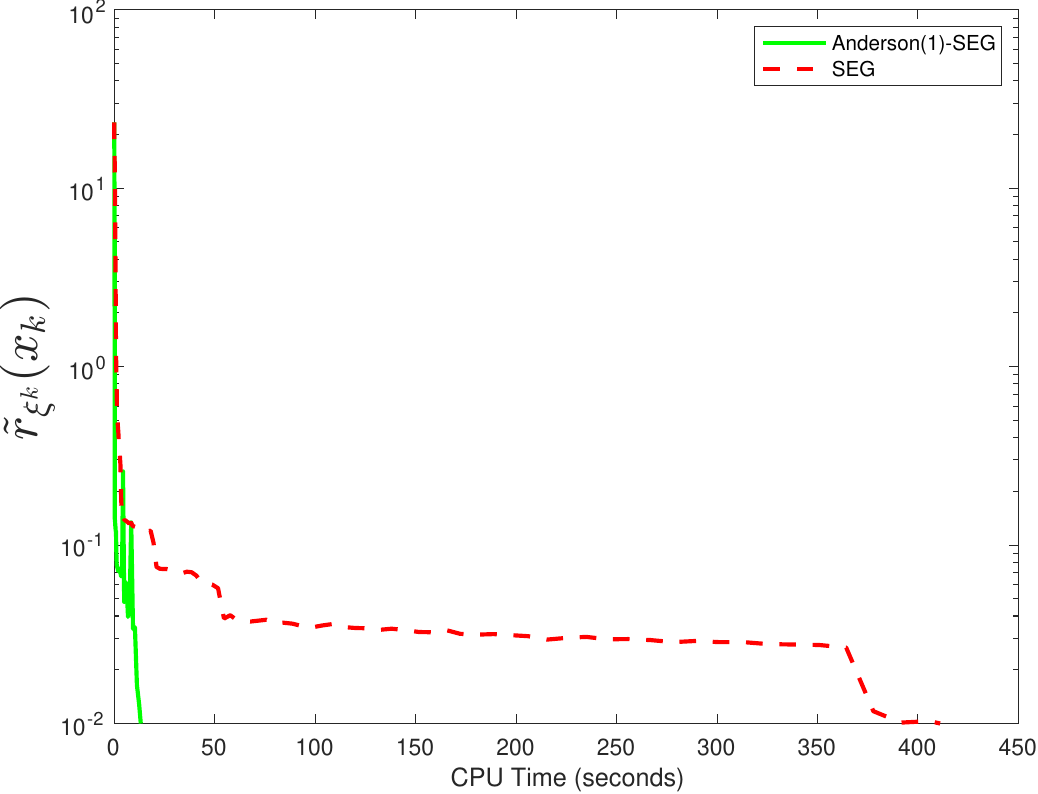}}
	\hfill
	\subfigure[$n=40$]{\includegraphics[width=0.48\textwidth]{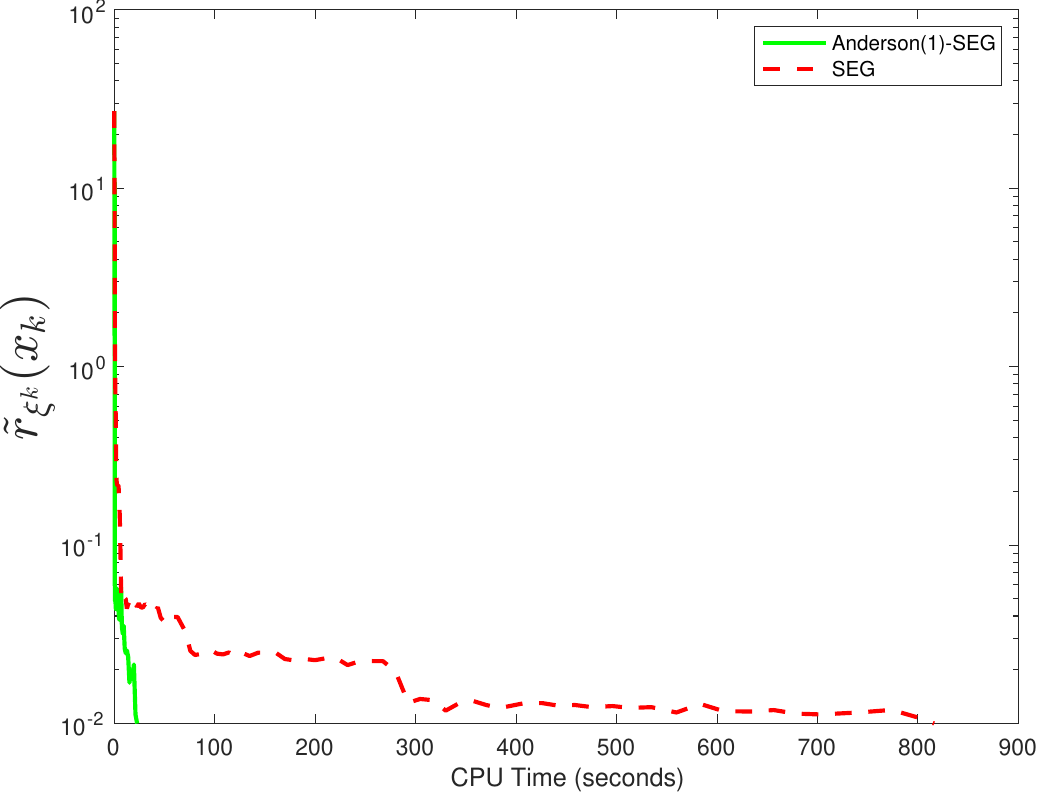}}
	
	\caption{Comparisons of the convergence behaviours of the Anderson(1)-SEG and the SEG algorithms for Example \ref{exam42}}
	\label{fig1020}
\end{figure}
We also perform the experiments where $S_k$
was set to $\mathcal{N}\lceil (k+\lambda)(\ln (k+\lambda))^{1+b} \rceil$ with $\gamma=0.9$, $\lambda=5$, $b=0.1$ and $\mathcal{N}=20$. 
Table \ref{t422} reports the average number of iterations and the average CPU time for both the algorithms across various dimensions, based on ten individual experimental runs. 
As shown in Table~\ref{t422}, across all tested dimensions, the Anderson(1)-SEG algorithm consistently requires fewer iterations and less CPU time than the SEG algorithm to meet the stopping criteria.
\begin{table}[H]
	\centering
	\label{table2}
	\caption{Comparisons of the two algorithms for Example \ref{exam42}}\label{t422}
	\begin{tabular}{c|cccc}
		\toprule
		Sec.(avr)\\ Iter.(avr) &  & Anderson(1)-SEG & SEG \\
		\midrule
		\multirow{2}{1cm}{$n=10$} &  & \textbf{0.2930} &  1.5853   \\
		&  & \textbf{ 24.9} (16.6)   &   62.4   \\
		\midrule
		\multirow{2}{1cm}{$n=20$} &  & \textbf{1.6086}  &    4.6915  \\
		& & \textbf{ 44.3} (25.9)  & 77 \\
		\midrule
		\multirow{2}{1cm}{$n=30$ } &  & \textbf{3.3112}  &  7.4201  \\
		& & \textbf{ 49.1} (29.6)   &  75.7 \\
		\midrule
		\multirow{2}{1cm}{$n=50$ } & & \textbf{14.4326} &  42.2535 \\
		& & \textbf{ 69.9} (39)   & 119.7\\
		\midrule
		\multirow{2}{1cm}{$n=100$ } &  &  \textbf{113.4942}  &  206.2057\\
		& & \textbf{96.5} (53.9)   & 124 \\
		\bottomrule
	\end{tabular}
\end{table}
Moreover, for Examples \ref{exam41} and \ref{exam42}, when $S_k$ is set to $\mathcal{N}\lceil (k+\lambda)(\ln (k+\lambda))^{1+b} \rceil$, where $\lambda=5$, $b=0.1$ and $\mathcal{N}=20$, Tables \ref{t411} and \ref{t422} show that the Anderson(1)-SEG algorithm is able to successfully solve both problems, even when $S_k$ is chosen to be smaller than the value specified in Assumption~\ref{a3}.
\subsection{Portfolio management}
\begin{example}
	%	Consider $n$ assets. Let $u\in\mathbb{R}^n$ denote the random returns of them.
	The Markowitz mean-variance model \cite{Pm} for portfolio selection can be formulated as
	\begin{equation}\label{objectvalue}
		\begin{split}
			& \min_{w\in\mathbb{R}^n}  \frac{1}{2}w^{\rm T}\Sigma w-\varrho^{\rm T}w
			\\  &~{\rm s.t.}~~ \mathbf{0}\leq w\leq a, ~\mathbf{e}^{\rm T}w=1,
		\end{split}
	\end{equation}
	where $w\in\mathbb{R}^n$ denotes the weights of the assets in the portfolio, $\varrho\in\mathbb{R}^n$, $\Sigma\in\mathbb{R}^{n\times n}$ are the mean vector and the covariance matrix of returns on assets, respectively.
	%$$\varrho=\mathbb{E}[u], \quad \Sigma=\mathbb{E}[(u-\varrho)(u-\varrho)^{\rm T}].$$
	$a\in\mathbb{R}_+^n$ is the upper bound enforced on $w$.
\end{example}
In the following experiments, we set $a=0.2\textbf{e}$. 
It is clear that problem \eqref{objectvalue} is equivalent to the linear complementarity problem denoted by ${\rm LCP}(\hat{q}, \hat{M})$, which is given by: 
\begin{equation}\label{SL}
	\hat{M}x+\hat{q}\geq \mathbf{0},~~~ x\geq \mathbf{0}, ~~~x^{\rm T}(\hat{M}x+\hat{q})=0,
\end{equation}
where
\[\hat{M}=
\begin{bmatrix} 
	\Sigma & -C^{\rm T} \\ 
	C & O 
\end{bmatrix},
C=\begin{bmatrix} 
	\textbf{e}^{\rm T}  \\ 
	-\textbf{e}^{\rm T} \\
	-I
\end{bmatrix},
\hat{q}=\begin{bmatrix} 
	-\varrho  \\ 
	-1  \\
	1  \\
	a 
\end{bmatrix},
x=\begin{bmatrix} 
	w \\ 
	y  
\end{bmatrix},
\] with $O$ denoting the zero matrix.
The dual variable $y$ is constrained to be in $\mathbb{R}_+^{n+2}.$
Moreover, the ${\rm LCP}(\hat{q}, \hat{M})$ in \eqref{SL} is equivalent to the ${\rm SVI}(X, H)$ with $H(x)=\hat{M}x+\hat{q}$ and $X=\mathbb{R}_+^n.$	
It can be verified that both Assumptions \ref{Minty} and \ref{a22} hold.
%Note that the covariance matrix $\Sigma$ is positive semi-definite, so $H$ is monotone, i.e., Assumption \ref{Minty} holds. According to the definitions of $\varrho$ and $\Sigma$, the random operator corresponding to $H$ is Lipschitz continuous, i.e., Assumption \ref{a22} is satisfied.

The stock data sets utilized in our experiments are drawn from two sources: the standard data sets available in the OR-library (referred to as Data 1-2) and those from the China A-share market (referred to as Data 3-6).
For the standard data sets, we utilize weekly stock prices from 1992 to 1997 for the DAX 100 Index (Germany) and the Nikkei Index (Japan). For Data sets 3 and 4, we employ daily closing prices from the China A-share market covering the period from 2010 to 2017. Additionally, for Data sets 5 and 6, we use daily closing prices for CSI300 index stocks from the China A-share market, spanning from 2015 to 2022. 

For the A-share and CSI300 index data, we remove stocks with more than 10$\%$  zero daily returns, resulting in 1,071 and 189 stocks as Data 3 and Data 5, respectively. 
We then rank these stocks by their average daily returns and select the top 150 and 100 stocks as Data 4 and Data 6, respectively.
A description of all the data sets can be found in Table \ref{datadescription}.
\vspace{20pt} % 增加10pt的垂直间距
\begin{table}[htbp]
	\centering
	\renewcommand{\arraystretch}{1.3}
	\setlength{\tabcolsep}{4pt}
		\caption{Description of the six real data sets}\label{datadescription}
	\begin{tabular}{c|c|c|c|c|c}
		\hline
		\multirow{1}{*}{\textbf{Data sets}} & \multirow{1}{*}{$n$} & \multirow{1}{*}{\textbf{Source}} & \multirow{1}{*}{\textbf{Index}} & 
		\multirow{1}{*}{$J$} & \multirow{1}{*}{\textbf{Description}}  \\
		\hline
		
		%Data 1 & 31 & Hong Kong & Hang Seng & 291 & weekly prices from 1992 to 1997  \\
		Data 1 & 85 & Germany & DAX 100 & 291 & weekly prices from 1992 to 1997 \\
		Data 2 & 225 & Japan & Nikkei 225 & 291 & weekly prices from 1992 to 1997  \\
		Data  3& 1071 &China  &A-share  & 1940 & daily prices from 2010 to 2017 \\
		Data 4 & 150 & China  & A-share & 1940 & daily prices from 2010 to 2017\\
		Data 5 & 189 & China & CSI300 & 1986 &daily prices from 2015 to 2022 \\
		Data 6 & 100 &  China& CSI300  & 1986 & daily prices from 2015 to 2022 \\
		\hline
	\end{tabular}
\end{table}

For each data set, let the stock closing prices be represented as $\{P_{j,i}:j\in[T], i\in [n]\}$. 
The $n$ represents the number of component assets, while the $T$ indicates the number of observations for those assets.
We compute the returns for each stock using the formula 
$r_{j,i}=\log \frac{P_{j+1,i}}{P_{j,i}}$ for $j\in [T-1], i\in[n]$.
This results in a return matrix $R\in\mathbb{R}^{(T-1)\times n}$ consisting of the elements $r_{j,i}$. Each data set is divided into two subsets in a 9:1 ratio, referred to as the training set and the testing set. The training set, known as the in-sample set, consists of the first 9/10 of the data and is used to calculate the optimal portfolio. The testing set, referred to as the out-of-sample set, includes the remaining data and is utilized to evaluate the performance of the resulting optimal portfolio.

Let $R\in\mathbb{R}^{(T-1)\times n}$
represent the return matrix of the assets over $[1, T-1]$. Based on the splitting ratio, we define $R_{in}\in\mathbb{R}^{T_{in}\times n}$ and $R_{out}\in\mathbb{R}^{T_{out}\times n}$, where 
$T_{in}=2\lceil 0.9(T-1)/2-1\rceil$ and $T_{out}=T-1-T_{in}$. Specifically, $R_{in}$ consists of the first $T_{in}$ rows of $R$, while $R_{out}$ comprises the remaining rows.

We use the Anderson(1)-SEG algorithm to solve the problem \eqref{SL} on the training set and find the optimal parameters $w$ for portfolio selection. We then compare its performance on the test set with that of the Naive portfolio (the naive evenly weighted portfolio, with all weights being $\frac{1}{n}$).
%{\color{blue}We utilize Anderson(1)-SEG algorithm to solve problem \eqref{SL} and find the optimal parameters $w$ for portfolio selection.} 
%{\color{blue}As in \cite{chen2018smoothing}, we consider the scenario in which $\Sigma$ is positive definite for risky assets.}
In practical applications, the out-of-sample portfolio performance is of greater interest, as it better reflects the portfolio's potential future returns. To enable comparison, we introduce two metrics to assess out-of-sample performance as follows:
\begin{enumerate}[\textbullet]
	\item the Sharpe ratio (SR) of the portfolio $w$, defined as $${\rm SR}(w) = \frac{\varrho_{out}^{\rm T}w}{\sqrt{w^{\rm T}\Sigma_{out} w}},$$
	where $\varrho_{out}$  is the mean of each column in $R_{out}$, and $\Sigma_{out}$ is the covariance matrix centered around $\varrho_{out}$;
	\item  
	the annualized return (AR) of the portfolio $w$, defined as
	$${\rm AR}(w)=\frac{{\rm CR}(w)}{T_w}\times 50\quad \text{or}\quad \frac{{\rm CR}(w)}{T_d}\times 250,$$ where
	${\rm CR}(w)$ represents the out-of-sample cumulative return for $w$, $T_w$ and $T_d$	refer to the number of weeks and days, respectively, in the out-of-sample dataset.
\end{enumerate}
In general, if a strategy has higher SR and AR, it is considered superior to other strategies.
%In general, the higher the values of the two criteria, the more effective the strategy is regarding out-of-sample CR and AR, respectively.

%Let $u\in\mathbb{R}^n$ denote the random returns of assets, then $\varrho=\mathbb{E}(u), \Sigma=\mathbb{E}[(u-\varrho)(u-\varrho)^{\rm T}]$.
When executing the Anderson(1)-SEG algorithm, the settings of some parameters in the algorithm are the same as in \eqref{para}. Moreover, we set $\gamma=0.1$ and
$S_k:=\min \left\{T_{in}/2, \mathcal{N}\lceil (k+\lambda)^2(\ln (k+\lambda))^{2+2b} \rceil\right\}$ with $\lambda=2+10^{-4}$, $b=10^{-4}$, and $\mathcal{N}=2$. 
%$$S_k= \min(T_{in}/2,  k + \lambda) $$
%with $\lambda=5$.  
We generate $\varrho\in\mathbb{R}^n$ and $\Sigma\in\mathbb{R}^{n\times n}$ in the following way:
$$\varrho_{\xi^{k}}=\text{mean}(R_{in}(T_{in}/2:T_{in}/2+S_k,:)); \quad \Sigma_{\xi^{k}}=\text{cov}(R_{in}(T_{in}/2:T_{in}/2+S_k,:));$$
$$\varrho_{\eta^{k}}=\text{mean}(R_{in}(1:S_k,:)); \quad \Sigma_{\eta^{k}}=\text{cov}(R_{in}(1:S_k, :)),$$ 
where $\{\xi^{k}\}$ and $\{\eta^{k}\}$ are random variables related to asset returns involved in each iteration.
Then, we have
\[\hat{M}_{\xi^{k}}=
\begin{bmatrix} 
	\Sigma_{\xi^{k}} & -C^{\rm T} \\ 
	C & O 
\end{bmatrix},\quad
\hat{q}_{\xi^{k}}=\begin{bmatrix} 
	-\varrho_{\xi^{k}}  \\ 
	-1  \\
	1  \\
	a 
\end{bmatrix}
\]
and 
\[\hat{M}_{\eta^{k}}=
\begin{bmatrix} 
	\Sigma_{\eta^{S_k}} & -C^{\rm T} \\ 
	C & O 
\end{bmatrix},\quad
\hat{q}_{\eta^{k}}=\begin{bmatrix} 
	-\varrho_{\eta^{k}}  \\ 
	-1  \\
	1  \\
	a 
\end{bmatrix}.
\]
%The stopping criterion for this example is set as
%$$\left\|r(\xi^k,x_k)\right\|=\left\|F_{1}(\xi^k,x_k)\right\|=\left\|P_{X}\left(x_k-  \left(\hat{M}(\xi^k)x_k+\hat{q}(\xi^k)\right)\right)-x_k\right\|<10^{-2}$$
%or when the number of iterations exceeds $1000.$  
The stopping criterion for this example is set as the number of iterations exceeding 2000. In the following figures and table, for simplicity, ‘Naive’ represents the average strategy, while ‘ASEG’ denotes the weights obtained by solving \eqref{SL} using the Anderson(1)-SEG algorithm. 
\begin{table}[H]
	\centering
	\renewcommand{\arraystretch}{1.3}
	\setlength{\tabcolsep}{4pt}
	\caption{Performance metrics for different data sets}\label{weightresults}
	\begin{tabular}{c|c|c|c|cc}
		\hline
		\multirow{2}{*}{\textbf{Data sets}} & \multicolumn{2}{c|}{\textbf{SR}} & \multicolumn{2}{c}{\textbf{AR}} \\
		\cline{2-5}
		& Naive & ASEG & Naive & ASEG \\
		\hline
		Data 1 & 0.2701 & \textbf{0.2744} &  9.5974e-04 & \textbf{0.0015} \\
		Data 2 & -0.2135 & \textbf{0.0716} &-0.0010 & \textbf{3.4974e-04} \\
		Data 3 &  -0.0492 & \textbf{0.0403} & -6.3285e-05 & \textbf{4.9256e-05}\\
		Data 4 & 0.0691 & \textbf{0.1029}& 8.6776e-05 & \textbf{1.4832e-04} \\
		Data 5 & 0.0506 & \textbf{0.0534} & 6.4001e-05 & \textbf{8.2941e-05} \\
		Data 6 & 0.0465 & \textbf{0.0536}&  6.5498e-05 & \textbf{8.3823e-05} \\
		\hline
	\end{tabular}
\end{table}
\begin{figure}[H]
	\centering
	\subfigure[]{
		\includegraphics[width=0.48\textwidth]{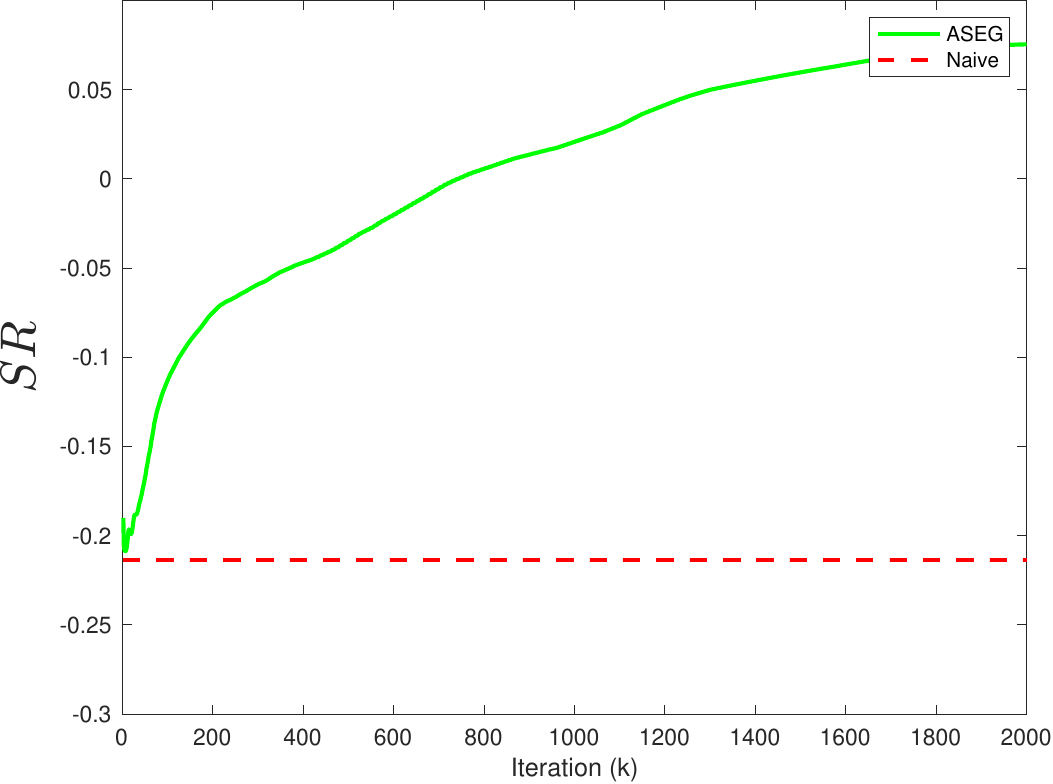}
		\label{figsr}
	}
	\hfill % 填充水平空间，使两图分开
	\subfigure[]{
		\includegraphics[width=0.48\textwidth]{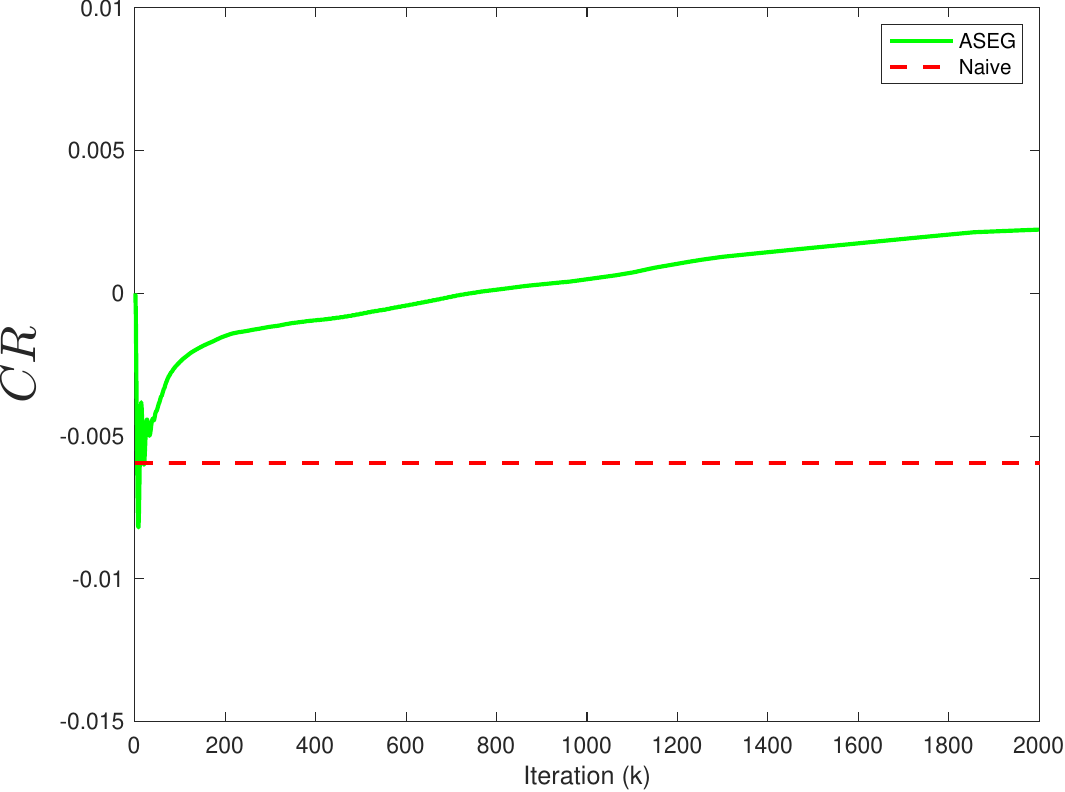}
		\label{fig3b}
	}
	\hfill % 填充水平空间，使两图分开
	\subfigure[]{
		\includegraphics[width=0.48\textwidth]{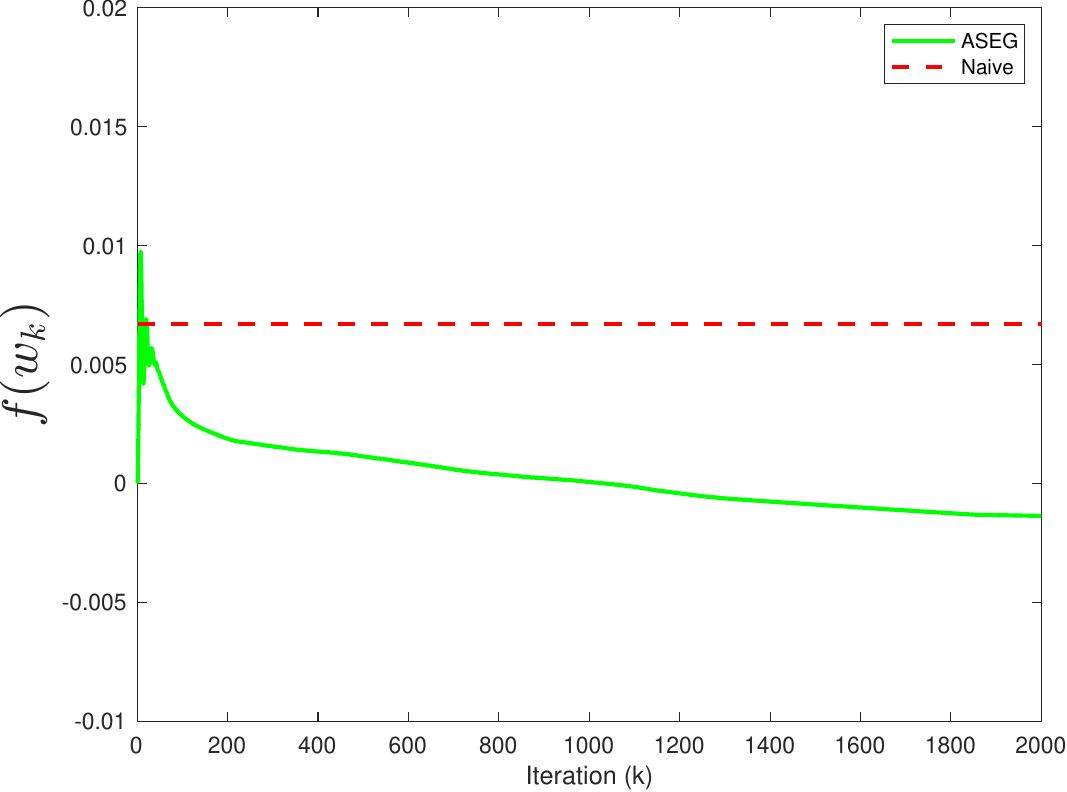}
		\label{fig3b}
	}
	\label{fig3}
	\caption{Convergence curves and comparison of SR and CR on the testing set of Data 2}\label{figcrsr}
\end{figure}
Table \ref{weightresults} presents the performance of the portfolio solved by the two methods on the test sets: results concerning SR and AR. 
According to Table \ref{weightresults}, it is evident that applying the Anderson(1)-SEG algorithm to solve \eqref{SL} yields a portfolio strategy with higher SR and AR compared to the average strategy across all data sets.
Moreover, Figure \ref{figcrsr} shows the curves of SR, CR, and the objective function value (abbreviated as $f$) in problem \eqref{objectvalue} on the testing set as the number of iterations increases for ASEG and Naive on Data 2.
\section{Conclusions}
This paper proposes an algorithm called Anderson(1)-SEG for solving ${\rm SVI}(X, H)$ where only the expected operator is assumed to be pseudomonotone. This algorithm employs the SA framework and combines the SEG and Anderson acceleration methods. To enhance robustness, the Anderson(1)-SEG algorithm integrates a line search strategy, eliminating the need for prior knowledge of the Lipschitz constant.
It is shown that, even though the sampled operators utilized in the algorithm lack pseudomonotonicity, the sequence generated by the Anderson(1)-SEG algorithm converges almost surely to a solution of problem ${\rm SVI}(X, H)$. Furthermore, a sublinear convergence rate of $\mathcal{O}(N^{-1})$ for the mean residual function is derived, and the algorithm's optimal oracle complexity is established. Numerical results indicate that the Anderson(1)-SEG algorithm achieves superior performance compared to the SEG method. Moreover, in the portfolio selection problem with real data, the solution obtained by the Anderson(1)-SEG algorithm outperforms the naive evenly weighted portfolio in terms of both SR and AR.
\vspace{5mm}
\\ \textbf{Data Availability}
The datasets generated during the current study are available from the corresponding author on reasonable request.

 \section*{Declarations}

 \textbf{Conflict of interest} The authors have not disclosed any competing interests.

%\begin{acknowledgements}
%\end{acknowledgements}



%\begin{thebibliography}{}
%%
%% and use \bibitem to create references. Consult the Instructions
%% for authors for reference list style.
%\end{thebibliography}

\end{document}